\documentclass[12pt]{article}
\usepackage{amssymb}
\usepackage{graphicx}

\usepackage{color}
\usepackage{latexsym}
\usepackage{amssymb,amsmath,amstext}
\usepackage{epsfig}
\usepackage{graphics}
\usepackage{subfigure}
\usepackage{euscript}
\usepackage{ifthen}
\usepackage{wrapfig}
\usepackage{amsthm}
\usepackage{multicol}
\usepackage{paralist}
\usepackage[colorlinks=true, urlcolor=blue, linkcolor=blue, citecolor=blue]{hyperref}
\usepackage[margin=1in]{geometry}

\usepackage{multicol}

\numberwithin{equation}{section}
\theoremstyle{plain}%
\newtheorem{theorem}{Theorem}
\numberwithin{theorem}{section}
\newtheorem{proposition}[theorem]{Proposition}
\newtheorem{example}[theorem]{Example}

\newtheorem{corollary}[theorem]{Corollary}

\newtheorem{remark}[theorem]{Remark}
\newtheorem{conjecture}[theorem]{Conjecture}

\newcommand{\C}{\mathbb{C}}

\newcommand{\PP}{\mathbb{P}}

\newcommand{\R}{\mathbb{R}}
\newcommand{\sS}{\mathcal{S}}
\newcommand{\sF}{\mathcal{F}}
\newcommand{\sP}{\mathcal{P}}
\newcommand{\sH}{\mathcal{H}}

\newcommand{\sL}{\mathcal{L}}
\newcommand{\sM}{\mathcal{M}}

\newcommand{\bS}{\mathbb{S}}

\allowdisplaybreaks

\date{}
\begin{document}

\title{\bf Maximum Likelihood for \\ Matrices with Rank Constraints}

\author{Jonathan Hauenstein \and Jose Rodriguez \and Bernd Sturmfels}

\maketitle

 \begin{abstract}
 \noindent
 Maximum likelihood estimation is a fundamental optimization problem in statistics. We study this
 problem on manifolds of matrices with bounded rank. These represent mixtures of
distributions of two independent discrete random variables. We determine the maximum likelihood degree
for a range of determinantal varieties, and we apply numerical algebraic geometry to compute all
critical points of their likelihood functions.
This led to the discovery of maximum likelihood duality between matrices
of complementary ranks, a result proved
subsequently by Draisma and Rodriguez.
         \end{abstract}

 \section{Introduction}

Maximum likelihood estimation (MLE) is a fundamental computational task in statistics.
A typical problem encountered in its applications is the occurrence of
multiple local maxima. In order to be certain that a global maximum
of the likelihood function has
been achieved, one needs to locate all solutions to a
system of polynomial equations.
In this paper we study these equations for  two discrete random variables,
 having $m$ and $n$ states respectively.
A joint probability distribution for  two such random variables is written as
an $m \times n$-matrix:
\begin{equation}
\label{eq:Pmatrix}
P \quad = \quad \begin{pmatrix}
p_{11} & p_{12} & \cdots & p_{1n} \\
p_{21} & p_{22} & \cdots & p_{2n} \\
 \vdots & \vdots & \ddots & \vdots \\
p_{m1} & p_{m2} & \cdots & p_{mn} \\
\end{pmatrix}.
\end{equation}
The entry $p_{ij}$ represents the probability that
the first variable is in state $i$ and the second is in state $j$.
Thus, the entries of $P$ are non-negative and their sum $p_{++}$ is $1$.
By a~statistical model, we mean a closed subset $\mathcal{M}$
of the probability simplex $\Delta_{mn-1}$ of all such matrices~$P$.

If  i.i.d.~samples are drawn from some $P$  then we summarize the data also in a matrix
\begin{equation}
\label{eq:Umatrix}
U \quad = \quad \begin{pmatrix}
u_{11} & u_{12} & \cdots & u_{1n} \\
u_{21} & u_{22} & \cdots & u_{2n} \\
 \vdots & \vdots & \ddots & \vdots \\
u_{m1} & u_{m2} & \cdots & u_{mn} \\
\end{pmatrix}.
\end{equation}
The entries of $U$ are non-negative integers whose sum is $u_{++}$.
As is customary in algebraic statistics \cite{LiAS, HKS, OpPro},
we write the {\em likelihood function} corresponding to the data
matrix $U$ as
\begin{equation}
\label{eq:loglike} \ell_U \quad = \quad
\frac{\prod_{i=1}^m \prod_{j=1}^n p_{ij}^{u_{ij}} }
{\bigl(\,\sum_{i=1}^m \sum_{j=1}^n p_{ij}\, \bigr)^{u_{++}}}.
\end{equation}
This formula defines a rational function on the
complex projective space $\PP^{mn-1}$
whose restriction to the simplex $\Delta_{mn-1}$ is
the usual likelihood function divided by a multinomial coefficient.
The MLE problem is to find the global maximum
of  $\ell_U$ over the model $\mathcal{M}$.

Our model of interest is
the set $\mathcal{M}_r$ of matrices $P$ of rank $\leq r$.
This is the intersection  of the
variety $\mathcal{V}_r \subset \mathbb{P}^{mn-1}$ defined by the $(r{+}1) \times (r{+}1)$-minors of~$P$
with $\Delta_{mn-1}$.
For generic $U$, the rational function $\ell_U$ has finitely many critical points
on the determinantal variety $\mathcal{V}_r$.
Their number is the {\em ML degree} of $\mathcal{V}_r$.
In this paper, we develop methods from numerical
algebraic geometry for computing all such critical points.
That computation enables us to
reliably find all local maxima of the likelihood function $\ell_U$
among positive points in $\mathcal{M}_r$.
Among the new results is the determination of the
bold face numbers in the following table.

\begin{theorem}
\label{thm:MLvalues}
The known values for the ML degrees of the determinantal varieties $\mathcal{V}_r$ are
\begin{equation}
\label{eq:MLvalues}
\begin{matrix}
        & (m,n) = & (3,3) & (3,4) & (3,5) & (4,4) & (4,5)  & (4,6) & (5,5) \\
r=1 & &      1 &   1   &   1 & 1 &  1 &  1     & 1 \\
r=2 & &      10 & 26 & {\bf 58} & {\bf 191} & {\bf 843} & {\bf 3119} & \, {\bf 6776} \\
r=3 & &        1         &1 &1 & {\bf 191} &  {\bf 843} &  {\bf 3119} & {\bf 61326}\, \\
r=4 & &       & & &1 &1 &1 & \, {\bf 6776} \\
r = 5 & &  & & & & & & \,\,1 \\
\end{matrix}
\end{equation}
\end{theorem}
The smaller numbers $10$ and $26$ had already been computed in
\cite[\S 5]{HKS}, but the symbolic computations using {\tt Singular} that were presented
in \cite{HKS} had failed beyond the size $3 \times 4$.

In 2005, the third author offered a cash prize of 100 Swiss Francs (cf.~\cite[\S 3]{OpPro})
for the solution of a particular $4 \times 4$-instance that was described in
\cite[Example 1.16]{ASCB}. That prize was won in 2008 by
Mingfu Zhu who solved this challenge in \cite{ZJG}.
See also \cite[Example 5.2]{Stei} for a solution using {\tt Singular},
and \cite{FHRZ} for a statistical perspective on this problem.
However, none of these papers had found the number $191$
of critical points for the $4 \times 4$ cases.
In the first version of this paper, we stated the conjecture
that the column symmetry among the ML degrees always holds.
This has subsequently been proven by Draisma and~Rodriguez:

\begin{theorem}[\cite{DR}]
\label{conj:duality}
If $m \leq n$ then the ML degrees for rank $r$ and for rank $m-r+1$ coincide.
\end{theorem}

Our findings might appeal also to those
interested in the topology of algebraic varieties.
For a variety $\mathcal{V}$ in $\PP^{mn-1}$,
let $\mathcal{V}^0$ denote the open subset  given by
 $\,p_{11} p_{12} \cdots p_{mn} p_{++} \not= 0$.
Huh \cite{Huh} recently proved that if $\mathcal{V}^0$ is smooth then
the ML degree of $\mathcal{V}$ is equal to the signed
  Euler characteristic of $\mathcal{V}^0$.
In our case, for $r \geq 2$, the open determinantal variety $\mathcal{V}^0_r$ is singular along
$\mathcal{V}^0_{r-1}$, but a suitably modified statement is expected to be true. It might be speculated that
the results in Theorems \ref{thm:MLvalues} and \ref{conj:duality}
will ultimately have a topological explanation.

The entries ``$1$'' of the table in (\ref{eq:MLvalues}) have easy explanations.
For $r=m$ we have
$\mathcal{V}_m = \PP^{mn-1}$ and the unique critical point of
the likelihood function $\ell_U$ is $P = \frac{1}{u_{++}} U$.
The first row of  (\ref{eq:MLvalues}) states that
the independence model $\mathcal{M}_1$ has ML degree $1$.
This fact is well-known to statisticians, as the rank $1$ matrix with entries
$ (u_{i+} u_{+j})/u_{++}^2$ is the unique critical point for $\ell_U$ on $\mathcal{V}^0_1$.
We found it instructive to derive this fact from Huh's~result~\cite[Theorem 1.(iii)]{Huh}:

\begin{example} \rm
Let $r=1$.
 The Segre variety $\mathcal{V}_1 = \PP^{m-1} \times \PP^{n-1}$ is smooth.
 Fix coordinates $(x_1: \cdots:x_m)$
 on $\PP^{m-1}$  and coordinates $(y_1: \cdots : y_n)$ on $\PP^{n-1}$.
 The open subset  $\mathcal{V}^0_1 $ consists of all
points in $\PP^{m-1} \times \PP^{n-1}$ with
$ x_1 x_2 \cdots x_m y_1 y_2 \cdots y_n (x_1{+}\cdots {+}x_m) (y_1 {+} \cdots {+} y_n) \not= 0$.
Hence
$$
\mathcal{V}^0_1 \quad = \quad
\hbox{($\PP^{m-1}$ minus $m+1$ hyperplanes) } \,\times \,
\hbox{($\PP^{n-1}$ minus $n+1$ hyperplanes)}.
$$
Each factor has signed Euler characteristic $1$, and hence so does their product.
\qed
\end{example}

This article is organized as follows.
In Section~\ref{sec:Equations}, we formulate the constraints
that characterize critical points of $\ell_U$ on $\mathcal{V}_r$
as a square system of polynomial equations. The specific formulation
in Theorem \ref{prop:LocalKernel} is one of our key contributions.
It is used to derive upper bounds in terms of
$m$, $n$, and $r$.
Theorem \ref{thm:MLSymmValues}
extends our results to the case of symmetric matrices,
and hence to mixtures of two identically distributed random variables.

Section~\ref{sec:NAG} is devoted to our computations using numerical
algebraic geometry.  This furnishes valuable new tools for practitioners of statistics who
are interested in exploring probability one
algorithms for computing the global~maximum of a given likelihood function.

In Section~\ref{sec:Conjectures}, we introduce a refined version of
Theorem \ref{conj:duality}, now
also proved in \cite{DR}, and we summarize the computational
evidence we had gathered to support it.
The Galois group computations in Proposition \ref{prop:galois}
might be of independent interest.
In Theorem \ref{thm:DiaNA}, we present a proof of
\cite[Conjecture~11]{ZJG} by means of certified numerical computations.

Section~\ref{sec:Stats} features the statistical view
on our approach, and we explain how it differs from running
the EM algorithm for discrete mixture models.
 The determinantal variety $\mathcal{V}_r$ is the Zariski closure of
the latent variable model for $r$-fold mixtures of independent variables.
They are equal in $\Delta_{mn-1}$
if and only if $r \leq 2$.
For $r \geq 3$ this takes us to
the real algebraic geometry problem, pioneered in \cite{MSS}, of distinguishing between
rank and non-negative~rank.

\section{Equations and bounds}\label{sec:Equations}

In this section, we present several formulations of the critical equations
for the likelihood function on the determinantal variety $\mathcal{V}_r = \{{\rm rank}(P) \leq r\}$.
We view $\mathcal{V}_r$
as an affine variety in the space of matrices $\C^{m \times n}$ and we
assume $m \leq n$.
Our main result is Theorem \ref{prop:LocalKernel}
which expresses our problem as a square system
of $mn$ polynomial equations in $mn$ unknowns.

An $m \times n$-matrix $P$ is a regular point in
the determinantal variety $\mathcal{V}_r$ if and only if
${\rm rank}(P)= r$. If this holds then the tangent
space $T_P$ is a linear subspace of dimension
$rn+rm-r^2$ in $\C^{m \times n}$, and its orthogonal complement
(with respect to the standard inner product)
 is a linear subspace $T_P^\perp$ of dimension  $(m-r)(n-r)$ in $\C^{m \times n}$.

 Our input is a strictly positive data matrix~$U$.
 We consider the logarithm of the likelihood function $\ell_U$ as in (\ref{eq:loglike}).
 The partial derivatives of the {\em log-likelihood function} $ {\rm log}(\ell_U)$ are then
\begin{equation}
\label{eq:logder1}
\frac{\partial {\rm log}(\ell_U)}{ \partial p_{ij} } \,\, = \,\, \frac{u_{ij} }{p_{ij}} - \frac{u_{++}}{p_{++}}  .
\end{equation}
By \cite[Proposition 3]{HKS}, a matrix $P$ of rank $r$ is a critical point for
 ${\rm log}(\ell_U)$ on $\mathcal{V}_r$ if and only if
the linear  subspace  $T_P^\perp$ contains
the $m \times n$-matrix whose $(i,j)$ entry is (\ref{eq:logder1}).
Hence the system of equations we seek to solve can be expressed in the following
{\em geometric formulation}:
\begin{equation}
\label{eq:formulation1}
{\rm rank}(P) = r \,,\quad \hbox{} \quad p_{++} = 1\,, \quad \hbox{and \
the matrix $\bigl( u_{ij}/p_{ij} - u_{++} \bigr) $ lies in $T_P^\perp$} .
\end{equation}
This is saying that the gradient of the objective function must
be orthogonal to the tangent space of the variety at a critical point
as in the elementary Lagrange multipliers method.
When translating (\ref{eq:formulation1}) into polynomial equations,
we need to make sure to exclude matrices $P$ of
rank strictly less than $r$, as these are singular points in $\mathcal{V}_r$.
We also need to exclude matrices $P$ with $p_{ij} = 0$ for some $(i,j)$.
These non-degeneracy conditions require some care.

In \cite{HKS}, the following formulation was used to represent our problem.
Let $J(P)$ denote the Jacobian matrix of the prime ideal defining $\mathcal{V}_r$. Since
that ideal is minimally generated by the $\binom{m}{r+1} \binom{n}{r+1}$
subdeterminants of format $(r+1) \times (r+1)$, the Jacobian $J(P)$
is a matrix of format $\,\binom{m}{r+1} \binom{n}{r+1}\times mn$ whose
entries are homogeneous polynomials of degree~$r$.
Let $[U]$ denote the matrix $U$ when written as a row vector of format
$1 \times mn$, and similarly $[P]$ is the vectorization of $P$.
We write ${\rm diag}[P]$ for the diagonal
$mn \times mn$-matrix with entries $p_{11}, p_{12} , \ldots, p_{mn}$.
The following extended Jacobian has
$2 + \binom{m}{r+1} \binom{n}{r+1}$ rows and $mn$ columns:
$$
\mathcal{J}(P) \quad = \quad
\begin{pmatrix}
[U] \\
[P] \\
J(P) \cdot {\rm diag}[P] \\
\end{pmatrix}.
$$
For a matrix $P$ of rank $r$, the
Jacobian $J(P)$ has rank $(m-r)(n-r) = {\rm codim}(\mathcal{V}_r)$.
The third condition in (\ref{eq:formulation1}) translates into the requirement
that the span of the first two rows intersects the rowspace of
$J(P) \cdot {\rm diag}[P] $. From this we derive the
{\em rank formulation}
\begin{equation}
\label{eq:formulation2}
 {\rm rank}(P) \leq r \quad \hbox{and} \quad
{\rm rank}( \mathcal{J}(P) ) \leq (m-r)(n-r) + 1 .
\end{equation}
This formulation of our problem is elegant and is adapted to
projective geometry in $\PP^{mn-1}$. In terms of equations, we simply take
the minors of size $r+1$ of the matrix $P$, and the
minors of size $(m-r)(n-r)+2$ of the matrix $\mathcal{J}(P)$.
However, this has two serious disadvantages: first, the number of minors
is enormous, and second, we must get rid of extraneous solutions
by saturation. Namely, to get rid of solutions $P$ with
${\rm rank}(P) \leq r-1$, we need to saturate by the
$r \times r$-minors of $P$, and to get rid of solutions
on the boundary, we need to saturate by the product of linear forms
$\,p_{11} p_{12} \cdots p_{mn} p_{++}$.
This was done symbolically in \cite[\S 4]{HKS}.

The calculation can be sped up a little bit by taking only $(m-r)(n-r)$ of the rows of $J(P)$,
while also imposing the non-homogeneous equation $p_{++}=1$.
Finally, we can replace the first two rows of $J(P)$ by a single row
$[U] - u_{++} [P]$ and require that the maximal minors
of the resulting $((m-r)(n-r)+1) \times mn$-matrix be zero.
This leads to some improvements but is still far from sufficient
to get to the full range of ML degrees reported in  Theorem \ref{thm:MLvalues}.

To get to those results, we pursue the following alternatives:
first, we introduce new unknowns which allow us to
replace the rank conditions by bilinear equations,
and, second, we  represent the subspace $\,T_P^\perp = {\rm rowspace}(J(P))\,$
using those same new unknowns.
Let $L $ be an $(m-r) \times m$-matrix of unknowns,
let $R $ be an $n \times (n-r)$-matrix of unknowns,
and $\Lambda = (\lambda_{ij})$ an $(n-r)\times (m-r)$-matrix of unknowns.
Then our {\em general kernel formulation}~is:
\begin{equation}
\label{eq:formulation3}
p_{++} = 1, \quad
L \cdot P = 0, \quad
P \cdot R = 0, \quad \hbox{and} \quad
 P \star (R \cdot \Lambda \cdot L)^T + u_{++} \cdot P = U .
\end{equation}
 Here $A \star B $ denotes the Hadamard (entry-wise) product of two matrices
of the same format.
If the rows of $L$ are linearly independent and the columns of $R$ are linearly
independent, then either of the conditions $L \cdot P = 0$ and $P \cdot R = 0$
suffice to imply that ${\rm rank}(P) \leq r$.

We now explain the last condition in (\ref{eq:formulation3}).
The space $T_P^\perp$ is spanned by the rank $1$ matrices
$(\rho_i\cdot\ell_j)^T$ where $\rho_i$ is the $i$-th column of $R$
and $\ell_j$ is the $j$-th row of $L$.
Then
$$
(R \cdot \Lambda \cdot L)^T \,\, = \,\,
\sum_{i=1}^{n-r} \sum_{j=1}^{m-r} \lambda_{ij}(\rho_i\cdot\ell_j)^T$$
is a general matrix in $T_P^\perp$.
 The matrix  $\bigl( u_{ij}/p_{ij} - u_{++} \bigr) $ in
 (\ref{eq:formulation1}) can be written as
\begin{equation}
\label{eq:logder2}
 P^{\star(-1)} \star U - u_{++}  \cdot {\bf 1} .
 \end{equation}
 Hence the last condition of (\ref{eq:formulation1}) is equivalent to
 saying (\ref{eq:logder2}) equals $(R \cdot \Lambda \cdot L)^T$ for some~$\Lambda$. We write this as
 $\,(R \cdot \Lambda \cdot L)^T  + u_{++}\cdot {\bf 1} =   P^{\star(-1)} \star U $. We take
Hadamard product of both sides with the matrix $P$ to get the last equation in (\ref{eq:formulation3}).
This operation is invertible since all entries~of~$U$ are non-zero.
Indeed, that last equation is $\, P \star \bigl( (R \cdot \Lambda \cdot L)^T + u_{++} \cdot {\bf 1} \bigr) = U$,
and if this holds then all $mn$ entries of the matrix $P$ must be~non-zero.

We conclude that (\ref{eq:formulation3}) is a correct formulation
of our problem provided we can ensure
$$ {\rm rank}(L) = m-r,\quad
{\rm rank}(R) = n-r,\quad \hbox{and}\quad
 {\rm rank}(P) = r. $$
We note that (\ref{eq:formulation3}) is highly
redundant as far as the number of variables is concerned.
There are several ways to reduce that number. For instance,
we can simply set $\lambda_{ij} = 1$ for all $i,j$.
 In addition, we can  either replace $L$ by a single row or replace $R$
by a single column. Even after these simplifications, the critical points
of $\ell_U$ on $\mathcal{V}_r$ are still represented faithfully.

After some experimentation, we found that the following simplification steps
lead to the best computational results.  Recall that $m \leq n$.
Let $P_1$ be an $r \times r$-matrix of unknowns,
let $R_1$ be an $r \times (n-r)$-matrix of unknowns, and
let $L_1$ be an $(m-r) \times r$-matrix of unknowns. The matrix
$\Lambda = (\lambda_{ij})$ is as before.
Using this notation, we take (\ref{eq:formulation3}) with
\begin{equation}\label{eq:formulation4}
L = \begin{pmatrix} L_1 &\! -I_{m-r}\end{pmatrix},\quad
P = \begin{pmatrix} P_1  &\! P_1 R_1 \\ L_1 P_1 &\! L_1 P_1 R_1 \end{pmatrix},\quad \hbox{and} \quad
R = \begin{pmatrix} R_1 \\ - I_{n-r} \end{pmatrix},
\end{equation}
where $I_{m-r}$ and $I_{n-r}$ are identity matrices.
We call (\ref{eq:formulation3}) with (\ref{eq:formulation4}) the
{\em local kernel formulation} of our problem.
Note that the constraints $L\cdot P = 0$, $P \cdot R = 0$,
${\rm rank}(L) = m-r$, and ${\rm rank}(R) = n-r$ are automatically
satisfied in this formulation.
The condition ${\rm rank}(P) = r$ is also implied for every
solution provided $U$ is generic.
Finally, the equation $p_{++} = 1$ can be removed from (\ref{eq:formulation3})
in this formulation since $p_{++}=1$ is equivalent to the sum of
all $mn$ equations given by $ P \star (R \cdot \Lambda \cdot L)^T + u_{++} \cdot P = U$.
By counting equations and unknowns, we now see that our system is a square system
consisting of $mn$ equations in $mn$ unknowns.

\begin{theorem}\label{prop:LocalKernel}
Let $U$ be a generic $m \times n$ data matrix with $m\leq n$.  The polynomial system
\begin{equation}
\label{eq:formulation5}
 P \star (R \cdot \Lambda \cdot L)^T + u_{++} \cdot P \,\,=\,\, U
\end{equation}
consists of $mn$ equations in $mn$ unknowns given by (\ref{eq:formulation4}).
It has finitely many complex solutions $(P_1,L_1,R_1,\Lambda)$, and the corresponding
$m \times n$-matrices
$P$ defined by (\ref{eq:formulation4}) are precisely the critical points of the
likelihood function $\ell_U$ on the determinantal variety $\mathcal{V}_r$.
\end{theorem}

Since the column sums of $P \star (R \cdot \Lambda \cdot L)^T$ are zero, we can further simplify $n$ equations.
For the first $m$ columns, we replace each entry on the diagonal with the column sum.
For the last $n-m$ columns, we replace the last entry in the column
with the column sum.

\begin{example}\label{ex:33} \rm
To illustrate the local kernel formulation (\ref{eq:formulation5}),
we consider $m=n=3$ with the two subcases $r=1$ and $r=2$.
Both have nine equations in nine unknowns.

\smallskip

\noindent {\em Subcase $r = 1$}:
The nine unknowns are the entries in the matrices
$$
L_1 = \begin{pmatrix} l_{11} \\ l_{21} \end{pmatrix}, \quad
P_1 =  \begin{pmatrix} p_{11} \end{pmatrix} , \quad
R_1 = \begin{pmatrix} r_{11} & r_{12}  \end{pmatrix}, \quad
\Lambda = \begin{pmatrix}
\lambda_{11} & \lambda_{12} \\
\lambda_{21} & \lambda_{22} \end{pmatrix},
$$
and the nine equations from (\ref{eq:formulation5}) take the form
$$
\begin{array}{rcl}
p_{11} (1 + l_{11} + l_{21}) &=& (u_{11} + u_{21} + u_{31})/u_{++} \\
p_{11} r_{11} (u_{++} - l_{11} \lambda_{11} - l_{21} \lambda_{12}) &=& u_{12} \\
p_{11} r_{12} (u_{++} - l_{11} \lambda_{21} - l_{21} \lambda_{22}) &=& u_{13} \\
p_{11} l_{11} (u_{++} - r_{11} \lambda_{11} - r_{12} \lambda_{21}) &=& u_{21} \\
p_{11} r_{11} (1 + l_{11} + l_{21}) &=& (u_{12} + u_{22} + u_{32})/u_{++} \\
p_{11} l_{11} r_{12} (\lambda_{21} + u_{++}) &=& u_{23} \\
p_{11} l_{21} (u_{++} - r_{11} \lambda_{12} + r_{12} \lambda_{22}) &=& u_{31} \\
p_{11} l_{21} r_{11} (\lambda_{12} + u_{++}) &=& u_{32} \\
p_{11} r_{12} (1 + l_{11} + l_{21}) &=& (u_{13} + u_{23} + u_{33})/u_{++}.
\end{array}~~~~~$$
This system has a unique solution which writes the
unknowns as rational functions in the~$u_{ij}$.

\bigskip

\noindent {\em Subcase $r = 2$}:
The nine unknowns are the entries in the matrices
$$
L_1  = \begin{pmatrix} l_{11}  & l_{12} \end{pmatrix},\quad
P_1 = \begin{pmatrix}
p_{11} & p_{12} \\
p_{21} & p_{22} \end{pmatrix},\quad
R_1 = \begin{pmatrix}
r_{11} \\ r_{21}
\end{pmatrix}, \quad
\Lambda = \begin{pmatrix} \lambda_{11} \end{pmatrix},
$$
and the nine equations take the form
$$
\begin{array}{rcl}
p_{11}(1+l_{11}) + p_{21}(1+l_{12}) &=& (u_{11} + u_{21} + u_{31})/u_{++} \\
p_{12} (l_{11} r_{21} \lambda_{11} + u_{++}) &=& u_{12} \\
(p_{11} r_{11} + p_{12} r_{21})(u_{++} - l_{11} \lambda_{11}) &=& u_{13} \\
p_{21} (l_{12} r_{11} \lambda_{11} + u_{++}) &=& u_{21} \\
p_{12}(1+l_{11}) + p_{22}(1+l_{12}) &=& (u_{12} + u_{22} + u_{32})/u_{++} \\
(p_{21} r_{11} + p_{22} r_{21}) (u_{++} - l_{12} \lambda_{11}) &=& u_{23} \\
(p_{11} l_{11} + p_{21} l_{12}) (u_{++} - r_{11}\lambda_{11}) &=& u_{31} \\
(p_{12} l_{11} + p_{22} l_{12}) (u_{++} - r_{21} \lambda_{11}) &=& u_{32} \\
(p_{11} r_{11} + p_{12} r_{21}) (1+l_{11}) + (p_{21} r_{11} + p_{22} r_{21}) (1 + l_{12}) &=& (u_{13} + u_{23} + u_{33})/u_{++}.
\end{array}~~~~~~~~$$
This system has ten complex solutions for a generic data matrix~$U$.
In  other words, the $9$ unknowns
$l_{\cdot \cdot}, p_{\cdot \cdot},  r_{\cdot \cdot}$ and
$\lambda_{11}$ are
algebraic functions of degree $10$ in  $u_{11},u_{12},\ldots,u_{33}$.
\qed
\end{example}

Upper bounds on the ML degree of $\mathcal{V}$ arise from our formulation.  The B\'ezout bound is
$$ 2^{r}\cdot 3^{n-r}\cdot 4^{n(m-1)}.$$
If we consider $(P_1,L_1,R_1,\Lambda)$ in the product space $\C^{r^2}\times\C^{r(m-r)}\times\C^{r(n-r)}\times\C^{(n-r)(m-r)}$,
our system consists of $r$ equations of degree $(1,1,0,0)$,
$n{-}r$ equations of degree $(1,1,1,0)$, and $n(m{-}1)$ equations of degree $(1,1,1,1)$.
The associated $4$-homogeneous B\'ezout bound is
the coefficient of the monomial $w^{r^2}\cdot x^{r(m-r)}\cdot y^{r(n-r)}\cdot z^{(n-r)(m-r)}$~in the expression
$$(w+x)^{r} \cdot (w+x+y)^{n-r} \cdot (w+x+y+z)^{n(m-1)}.$$

A refinement of the $4$-homogeneous bound using the fact that
each polynomial only depends upon a subset of the variables
yields a {\em linear product bound} \cite{LinearProduct}.
Finally, the {\em polyhedral root count} exploits the sparsity of the monomials
in our system.  We computed
the polyhedral bound for various cases using {\tt MixedVol} \cite{MixedVol} in {\tt PHC}~\cite{PHC}.
All of the aforementioned bounds are presented
in Table \ref{tab:Bounds} for selected values of $m$, $n$, and~$r$.
When solving a polynomial system using homotopies built from these bounds,
one must balance the added computational cost required for the tighter bound with the
computational savings arising from that bound.

\begin{table}
  \centering
{\scriptsize
  \begin{tabular}{ccccccc}
  $(m,n,r)$  & $(3,3,1)$ & $(3,3,2)$ & $(3,4,1)$ & $(3,4,2)$ & $(3,5,1)$ & $(3,5,2)$\\
  \hline
  B\'ezout       & 73728 & 49152 & 3538944 & 2359296 & 169869312 & 113246208\\
  $4$-hom        & 270   & 1350  & 840     & 29400   & 2025      & 378000 \\
  linear product & 172   & 1018  & 374     & 20844   & 650       & 68586 \\
  polyhedral     & 6     & 53    & 10      & 472     & 15        & 2724 \\
  ML Degree      & 1     & 10    & 1       & 26      & 1         & 58 \\
  \hline \\
  $(m,n,r)$ & $(4,4,1)$ & $(4,4,2)$ & $(4,4,3)$ & $(4,5,1)$ & $(4,5,2)$ & $(4,5,3)$\\
  \hline
  B\'ezout       & 905969664 & 603979776 & 402653184 & 173946175488 & 115964116992 & 77309411328 \\
  $4$-hom        & 17600     & 7276500   & 580800    & 63700        & 323723400    & 115615500 \\
  linear product & 5690      & 4791168   & 224598    & 13560        & 165869606    & 58335270 \\
  polyhedral     & 20        & 15280     & 2847      & 35           & 241218       & 145273 \\
  ML Degree      & 1         & 191       & 191       & 1            & 843          & 843 \\
  \hline
  \end{tabular}
}
\caption{Comparison of upper bounds for selected $(m,n,r)$}
\label{tab:Bounds}
\end{table}

\medskip

We close this section by discussing rank constraints on symmetric matrices of the form
\begin{equation}
\label{eq:Pmatrixsym}
P \quad = \quad \begin{pmatrix}
2p_{11} & p_{12} & p_{13} & \cdots & p_{1n} \\
p_{12} & 2 p_{22} & p_{23} & \cdots & p_{2n} \\
p_{13} & p_{23} & 2 p_{33} & \cdots & p_{3n} \\
 \vdots & \vdots & \vdots & \ddots & \vdots \\
p_{1n} & p_{2n} & p_{3n} & \cdots & 2 p_{nn} \\
\end{pmatrix}.
\end{equation}
The case $n = 3$ was treated in \cite[Example 12]{HKS}
where its ML degree was found to be $6$.
It is essential that the unknowns $p_{ii}$
on the diagonal are multiplied by $2$
before imposing the rank constraints.
The matrices (\ref{eq:Pmatrixsym}) of rank one form
 a Veronese variety in $\PP^{(n+2)(n-1)/2}$.
This variety has ML degree $1$ and
represents the independence model for two
identically distributed random variables on $n$ states.
The case $n = 2$ is the Hardy-Weinberg curve
\cite[Figure 3.1]{ASCB}.
Larger ranks $r$ correspond to the secant varieties of this Veronese variety.

\begin{theorem}
\label{thm:MLSymmValues}
The known values for the ML degrees of rank $r$ symmetric matrices (\ref{eq:Pmatrixsym})~are
\begin{equation}
\label{eq:MLvaluessym}
\begin{matrix}
        & n = & 3 & 4 & 5 & 6 \\
r=1 & &         1 & 1 & 1 & 1  \\
r=2 & &         6 & {\bf 37} & {\bf 270} & {\bf 2341}\\
r=3 & &         1 & {\bf 37} & {\bf 1394}& \\
r=4 & &           &     1     & {\bf 270} & \\
r=5 & &           &          &    1       & {\bf 2341}\\
\end{matrix}
\end{equation}
\end{theorem}

\smallskip

Our input is a strictly positive symmetric $n\times n$-matrix $U$.
The likelihood function equals
\begin{equation}
\label{eq:logsymmlike} \ell_U \quad = \quad
\frac{\prod_{i\leq j} p_{ij}^{u_{ij}} }
{\bigl(\,\sum_{i\leq j} p_{ij}\,\bigr)^{\sum_{i\leq j} u_{ij}}}.
\end{equation}
In the statistical context, when the sum of the $p_{ij}$ entries equals $1$, we have
\begin{equation}
\label{eq:logsymmder1}
\frac{\partial {\rm log}(\ell_U)}{ \partial p_{ij} } \,\, = \,\, \frac{u_{ij}}{p_{ij}} - \sum_{i\leq j} u_{ij} .
\end{equation}
We compute the critical points on the variety of rank $r$ matrices (\ref{eq:Pmatrixsym})
by adapting the formulation in Theorem \ref{prop:LocalKernel}.
Let $P_1$ be a symmetric $r\times r$-matrix of unknowns
where the diagonal entries are multiplied by $2$ similar to (\ref{eq:Pmatrixsym}),
let $L_1$ be an $(n-r)\times r$-matrix of unknowns,
and $\Lambda$ be a symmetric $(n-r)\times(n-r)$-matrix.
Following (\ref{eq:formulation4}), we define
\begin{equation}\label{eq:SymmFormulation}
L = \begin{pmatrix} L_1 &\! -I_{m-r}\end{pmatrix} \quad \hbox{and} \quad
P = \begin{pmatrix} P_1  &\! P_1 L_1^T \\ L_1 P_1 &\! L_1 P_1 L_1^T \end{pmatrix}.
\end{equation}

To account for the $p_{ii}$'s not being multiplied by $2$ in the likelihood function,
let $D$ be the $n\times n$-matrix whose diagonal entries are $2$ and off-diagonal entries are $1$.
The {\em symmetric local kernel formulation} is the square system consisting of the upper triangular part of
\begin{equation}\label{eq:SymmFormulation1}
P\star (L^T\cdot\Lambda\cdot L) + \sum_{i\leq j} u_{ij} \cdot P \,\, = \,\, D\star U.
\end{equation}
This is a system of $n(n+1)/2$ equations in $n(n+1)/2$ unknowns.  Similar to the local kernel formulation,
the column sums of $P\star(L^T\cdot\Lambda\cdot L)$ are zero. Hence
(\ref{eq:SymmFormulation1}) implies $\sum_{i \leq j} p_{ij} = 1$.  We use
this fact to replace the diagonal entries in (\ref{eq:SymmFormulation1}) with the corresponding column sum.

\begin{example}\label{ex:33symm} \rm
We illustrate the symmetric local kernel formulation (\ref{eq:SymmFormulation1})
for the two subcases $r=1,2$ when $n=3$.
Both have $6$ equations in $6$ unknowns.  Here, $u_{++} = \sum_{i\leq j} u_{ij}$.

\smallskip

\noindent {\em Subcase $r = 1$}:
The six unknowns arise from the entries in the matrices
$$
L_1 = \begin{pmatrix} l_{11} \\ l_{21} \end{pmatrix}, \quad
P_1 =  \begin{pmatrix} 2p_{11} \end{pmatrix} , \quad
\Lambda = \begin{pmatrix}
\lambda_{11} & \lambda_{12} \\
\lambda_{12} & \lambda_{22} \end{pmatrix},
$$
and the six equations take the form
$$
\begin{array}{rcl}
2 p_{11}(1+l_{11}+l_{21}) &=& (2u_{11} + u_{12} + u_{13})/u_{++} \\
2 p_{11} l_{11} (u_{++} - l_{11}\lambda_{11} - l_{21}\lambda_{12}) &=& u_{12} \\
2 p_{11} l_{21} (u_{++} - l_{11}\lambda_{12} - l_{21}\lambda_{22}) &=& u_{13} \\
2 p_{11} l_{11} (1+l_{11}+l_{21}) &=& (u_{12} + 2u_{22} + u_{23})/u_{++} \\
2 p_{11} l_{11} l_{21} (\lambda_{12} + u_{++}) &=& u_{23} \\
2 p_{11} l_{21} (1 + l_{11} + l_{21}) &=& (u_{13} + u_{23} + 2u_{33})/u_{++}.
\end{array}~~~~~$$
This system has a unique solution which writes the
  unknowns as  rational functions in the~$u_{ij}$.

\bigskip

\noindent {\em Subcase $r = 2$}:
The six unknowns arise from the entries in the matrices
$$
L_1  = \begin{pmatrix} l_{11}  & l_{12} \end{pmatrix},\quad
P_1 = \begin{pmatrix}
2p_{11} & p_{12} \\
p_{12} & 2p_{22} \end{pmatrix},\quad
\Lambda = \begin{pmatrix} \lambda_{11} \end{pmatrix},
$$
and the six equations take the form
$$
\begin{array}{rcl}
2p_{11}(1+l_{11}) + p_{12}(1+l_{12}) &=& (2u_{11} + u_{12} + u_{13})/u_{++} \\
p_{12} (l_{11} l_{12} \lambda_{11} + u_{++}) &=& u_{12} \\
(2 p_{11} l_{11} + p_{12} l_{12}) (u_{++} - l_{11} \lambda_{11}) &=& u_{13} \\
p_{12}(1+l_{11}) + 2p_{22}(1+l_{12}) &=& (u_{12} + 2u_{22} + u_{23})/u_{++} \\
(p_{12}l_{11} + 2p_{22}l_{12})(u_{++} - l_{12}\lambda_{11}) &=& u_{23} \\
(2p_{11}l_{11} + p_{12}l_{12})(1+l_{11}) + (p_{12} l_{11} + 2 p_{22} l_{12})(1+l_{12}) &=& (u_{13} + u_{23} + 2u_{33})/u_{++}.

\end{array}~~~~~~~~$$
This system has six complex solutions for a general data matrix~$U$.
In the other words, the $6$ unknowns
$l_{\cdot \cdot}, p_{\cdot \cdot}$, and $\lambda_{11}$
are algebraic functions of degree $6$ in  $u_{11},u_{12},\ldots,u_{33}$.
\qed
\end{example}

Here is the symmetric version of Theorem \ref{conj:duality},
as suggested by Theorem \ref{thm:MLSymmValues}:

\begin{theorem}[Draisma and Rodriguez \cite{DR}]
\label{conj:dualitysym}
The ML degree for symmetric
$n \times n$-matrices (\ref{eq:Pmatrixsym}) of rank $r$ is equal to  the ML degree for
 symmetric $n \times n$-matrices (\ref{eq:Pmatrixsym}) of rank $n-r+1$.
\end{theorem}

This was  stated as a conjecture in the first version of this paper,
and proved later in~\cite{DR}.

\section{Solutions using numerical algebraic geometry}\label{sec:NAG}

 Theorems \ref{thm:MLvalues} and \ref{thm:MLSymmValues}
document considerable advances relative to  the
 computational results found earlier
 in \cite[\S 5]{HKS}. In this project, we
used numerical algebraic geometry \cite{BHSW}
to compute the ML degrees by solving
the local kernel formulation (\ref{eq:formulation5}) which we explain in this section.

The statistical problem addressed here is to find the global maximum of
a likelihood function $\ell_U$ over a matrix model $\mathcal{M}$
given by rank constraints.
For this class of problems, the use of numerical algebraic geometry
has the following significant advantage over symbolic computations.
After having solved the likelihood equations only once, for one
generic data matrix $U_0$, all subsequent computations
for other data matrices $U$ are much faster.
Numerical homotopy continuation will start from the critical points
of $\ell_{U_0}$ and transform them into the
critical points of $\ell_{U}$.
Intuitively speaking, for a fixed model $\mathcal{M}$,
{\em the homotopy amounts to changing the data}.
We believe that our methodology will be useful for a wider range of maximum likelihood problems than those treated here,
and we decidedly agree with the statement in
\cite[\S 5]{BR} that
{\em ``... homotopy continuation algorithms often provide substantial advantages over iterative methods commonly used in statistics''}.

We discuss below two options for the preprocessing stage
of solving the local kernel formulation (\ref{eq:formulation5})
for generic $U_0$.  The first option is to use a single
homotopy built from an upper bound discussed in Section~\ref{sec:Equations},
most notably a polyhedral homotopy built from the polyhedral root count.
The second option is to use a sequence of homotopies that intersect
the hypersurfaces corresponding to each equation, most notably via regeneration \cite{HSW}.


Parallel computation is an essential feature of numerical algebraic geometry.
Both preprocessing, by solving a generic data set once,
and each subsequent solve for given specific data can be performed in parallel.
In our case, we used a 64-bit Linux cluster with $160$ processors
to perform the computations summarized in Table \ref{tab:RunTime}
which tracked each path on a separate processor.
For instance, for  $(m,n,r) = (4,5,2)$, there are $843$ paths,
to be distributed among the $160$ processors.
Using adaptive precision~\cite{AMP}, this takes $20$ seconds
while the same computation performed sequentially takes about $20$ minutes on a typical~laptop.

\begin{example} \rm
The following data matrix is attributed to
the fictional character DiaNA  in \cite[Example 1.3]{ASCB}.
It represents her alignment of two DNA sequences of length $u_{++} = 40$:
 $$ U \quad = \quad \begin{pmatrix}
4 & 2 & 2 & 2 \\
2 & 4 & 2 & 2 \\
2 & 2 & 4 & 2 \\
2 & 2 & 2 & 4
\end{pmatrix}.
$$
According to Table \ref{tab:RunTime}, it took
$257$ seconds to solve the first instance for $(m,n,r) = (4,4,2)$, but now
every subsequent run takes only $4$ seconds.
In that solving step, the integers $u_{ij}$ become parameters over the complex numbers.
For DiaNA's data matrix $U$, the $191$ complex critical points degenerate to
$25$ real critical points, each of which is positive, and $166$~nonreal critical points.
See Theorem~\ref{thm:DiaNA} for additional information regarding
the critical points.
\qed
\end{example}

\begin{table}
  \centering
  \begin{tabular}{cccccccc}
  $(m,n,r)$ & $(4,4,2)$ & $(4,4,3)$ & $(4,5,2)$ & $(4,5,3)$ & $(5,5,2)$  & $(5,5,4)$ \\
  \hline
  Preprocessing & 257 & 427 & 1938 & 2902 & 348555 &  146952   \\
  Solving & 4 & 4 & 20 & 20 & 83 &  83  \\
\end{tabular}
\caption{
Comparison of running times for preprocessing and
 subsequent solving (in seconds)}\label{tab:RunTime}
\end{table}

Three advantages of the local kernel formulation (\ref{eq:formulation5}) are that
it is a square system with polynomials of degree at most $4$,
it is sparse in terms of the number of monomials appearing, and
it has a natural product structure.
These structures are clearly visible from the systems in Example~\ref{ex:33}, and
they are used to derive the smaller upper bounds in Table \ref{tab:Bounds}.
In what follows, we shall describe our preprocessing and how we can use its output to easily compute all
critical points of $\ell_U$ for a given data matrix $U$.
We also analyze some specific examples.
An introduction to numerical algebraic geometry and homotopy continuation can be found in \cite{SW}
and more details using {\tt Bertini} to perform these
computations in the~forthcoming~book~\cite{BHSW}.

\smallskip

For a square polynomial system $F$, {\em basic homotopy continuation} computes a finite
set $\sS$ of complex roots of $F$ which contains all  isolated roots.  Here, ``computes $\sS$'' means
numerically computing the coordinates of each point in $\sS$,
and to be able to approximate these  to arbitrary accuracy.
Numerical approximations to nonsingular solutions can be certified
using the software {\tt alphaCertified} \cite{alphaCertified}. This certification
can also determine if the solution is real or positive.
To compute $\sS$, we first construct a family of polynomial systems $\sF$ containing $F$ and
then compute the isolated roots for a sufficiently general $G\in\sF$.  Finally,
one tracks the solution paths starting with the isolated roots as $G$ deforms to $F$ inside~$\sF$.

Fix $(m,n,r)$ and let $\sF := \sF_{m,n,r}$ be the family of polynomial systems
(\ref{eq:formulation5}) for $U\in\C^{m\times n}$.
The generic root count on $\sF$ is the ML degree of $\mathcal{V}_r$.
In particular, for any generic $U_0 \in\C^{m\times n}$
the number of roots of the corresponding system $F_{U_0} \in\sF$ is
the ML degree of $\mathcal{V}_r$.  Suppose further that we know the roots of $F_{U_0}$.
Then, for any matrix $U\in\C^{m\times n}$, we can compute the isolated roots of the
corresponding polynomial system $F_U$ by tracking the ML degree number
of solutions paths starting with the roots of $F_{U_0}$ as $U_0$ and $F_{U_0}$ deform to $U$ and $F_U$.

Since the family $\sF$ is parameterized by the linear space $\C^{m\times n}\cong\mathbb{R}^{2mn}$,
we can connect $U_0$ to $U$ along a line segment.
If $U_0$ is not in a sufficiently general position with respect to $U$, e.g., both real,
this segment may contain matrices for which the corresponding
system has a root count that is different from the ML degree.  To avoid this,
we apply the  {\em gamma trick} of \cite{MS}.
For $\gamma\in\bS^1 \subset \C^*$, the trick deforms from $U_0$ to $U$ along the arc parameterized~by
\begin{equation}
\label{eq:gammatrick}
 \frac{\gamma t}{1+(\gamma-1) t} \cdot U_0 \,+
\, \frac{1-t}{1+(\gamma-1)t} \cdot U \quad \hbox{for} \quad t\in[0,1].
\end{equation}
For all but finitely many values $\gamma\in\bS^1$,
the root count for the corresponding polynomial system along this arc, except possibly at $U$ when $t = 0$,
is the ML degree.

We conclude our discussion on deforming from a known set of critical points with a
practical issue.  Due to choices of affine patches, the local kernel formulation
(\ref{eq:formulation5}), as written, is not suitable for a nongeneric data matrix $U$.
Once given a data matrix $U$, we simply choose random affine patches as in \cite{BHPS}.
Let $O_1,O_2\in\R^{r\times r}$, $O_3\in\R^{m\times m}$, and $O_4\in\R^{n\times n}$ be random orthogonal matrices
and $L_1$, $P_1$, $R_1$, and $\Lambda$ be as before.  Then, we use (\ref{eq:formulation5}) with
$$
L = O_1 \cdot \begin{pmatrix} L_1 &\! -I_{m-r}\end{pmatrix} \cdot O_3^T,\,\,\,
P = O_3 \cdot \begin{pmatrix} P_1  &\! P_1 R_1 \\ L_1 P_1 &\! L_1 P_1 R_1 \end{pmatrix} \cdot O_4^T,\, \hbox{~and~} \,
R = O_4 \cdot \begin{pmatrix} R_1 \\ - I_{n-r} \end{pmatrix} \cdot O_2^T.
$$
The homotopy (\ref{eq:gammatrick})
quickly computes the isolated critical points
for any given data matrix $U$  provided that we already know
the critical points for a sufficiently general data matrix~$U_0$.

\smallskip

We now discuss the two options for {\em preprocessing} mentioned above,
namely polyhedral homotopies and regeneration.
A summary of our computations with these two methods, now using serial processing
with double precision, are presented in Table~\ref{tab:PreprocessTime}.
The last pair of entries suggest that the two methods exhibit complementary 
behavior with respect to the duality of Theorem \ref{conj:duality}. In both cases,
$191$ roots are found, and these are essentially the same roots,
by Theorem \ref{conj:Constant} below. For instance, using polyhedral homotopy,
the rank $2$ case can be solved in  $1869$ seconds
and then we may read off the solutions for rank $3$ using~(\ref{eq:strongduality}).

The first approach to solve the equations for $U_0$ is to use basic
homotopy continuation in the family $\sP$ of polynomial systems
that arise from some relevant structure.
The generic root count on $\sP$ constructed from various structures
are presented in Table \ref{tab:Bounds}.
After computing the roots for a general element of $\sP$,
we return to basic homotopy continuation for computing the roots of $F_{U_0}$.
Table~\ref{tab:PreprocessTime} summarizes using a polyhedral approach
implemented in {\tt PHC} \cite{PHC}
where the family $\sP$ is constructed based on the
Newton polytopes of the given equations.

The second approach is based on intersecting the given hypersurfaces iteratively.
This can be
advantageous when the degree of the intersection is significantly less than the
product of the degrees. To be explicit,
if $\sS$ is a pure $k$-dimensional variety ($k > 0$) and $\sH$ is a hypersurface,
intersection approaches can be advantageous when
the degree of the pure $(k-1)$-dimensional part of $\sS\cap\sH$
is less than $\deg \sS \cdot \deg \sH$.
Regeneration is an intersection approach that builds from
a product structure of the given system. We shall now discuss~this.

\begin{table}
  \centering
  \begin{tabular}{ccccccc}
  $(m,n,r)$ & $(3,3,2)$ & $(3,4,2)$ & $(3,5,2)$ & $(4,4,2)$ & $(4,4,3)$ \\
  \hline
  Polyhedral using {\tt PHC} & 4 & 120 & 2017 & 23843 & 1869 \\
  Regeneration using {\tt Bertini} & 6 & 61 & 188 & 2348 & 7207 \\
\end{tabular}
\caption{
Running times for preprocessing in serial using double precision (in seconds)}\label{tab:PreprocessTime}
\end{table}

We first consider the classical idea of solving polynomial systems
using successive intersections and then discuss
how to build from a product structure.
Consider $N$ polynomials $f_1,\dots,f_N$ in $N$ variables,
defining hypersurfaces $\sH_1,\dots,\sH_N$.
One advantage of a square system is that the isolated solutions of $f_1 = \cdots = f_N = 0$ arise by computing
the codimension $i$ components of $\sH_1\cap\cdots\cap\sH_i$ sequentially for $i = 1,2,\dots,N$.
In fact, every codimension $i+1$ component of $\sH_1\cap\cdots \cap \sH_i\cap\sH_{i+1}$
arises as the intersection of a codimension $i$ component $C$ of
$\sH_1\cap\cdots\cap\sH_i$ and the hypersurface $\sH_{i+1}$, where $C$ is
not contained in $\sH_{i+1}$.

The use of the product structure arises from intersecting an algebraic set of
pure codimension $i$ with a linear space of dimension $i$ yielding finitely many points.
The first step is a hypersurface intersected with a line.
If $\sL_2,\dots,\sL_N$ are general hyperplanes, the hypersurface
$\sH_1$ is represented by the isolated points in $\sH_1\cap\sL_2\cap\cdots\cap\sL_N$.
Such points can be computed by solving a univariate polynomial, namely
$f_1$ restricted to the line $\sL_2\cap\cdots\cap\sL_N$.
Let $1\leq i<N$ and $C_i$ be the pure one-dimensional component of
$\sH_1\cap\cdots\cap\sH_i\cap\sL_{i+2}\cap\cdots\cap\sL_N$.
Now, {\em basic regeneration} computes $C_i\cap\sH_{i+1}$ from $C_i\cap\sL_{i+1}$ as follows.
Let $\sM_1,\ldots,\sM_k$ be hyperplanes defined by
sufficiently general linear polynomials $\ell_1,\ldots,\ell_k$
that represent a linear product decomposition of $f_{i+1}$.
Let $\sM = \bigcup_{j=1}^k \sM_j$.   Basic homotopy continuation
computes $C_i\cap\sM_j$ from $C_i\cap\sL_{i+1}$ for $j = 1,\dots,k$.
Their union is $C_i\cap\sM$.  Applying basic homotopy continuation once more
yields $C_i\cap\sH_{i+1}$ by deforming~from~$C_i\cap\sM$.

For the preprocessing approaches above, we can certify that the
set of approximations obtained correspond to distinct solutions using
{\tt alphaCertified}.  At each stage of the regeneration and at the
end of the computation, we can perform one additional
test to confirm that we have obtained all of the solutions: the trace test \cite{SVW}.
During regeneration, the centroid of the solutions must move linearly as
the hyperplane $\sL_N$ is moved linearly.
Moreover, the centroid of the critical $m\times n$-matrices
must move linearly as the data matrix $U$ moves linearly.
With these tests, we are able to claim, with high probability,
that our initial randomly selected data matrix $U_0$
was sufficiently generic, and Theorems~\ref{thm:MLvalues} and~\ref{thm:MLSymmValues}~hold.

\smallskip

After computing the positive critical points for a given data matrix~$U$, we
identify the local
maximizers by analyzing the Hessian of the corresponding Lagrangian function, namely
$$L(P,\lambda) \,\,=\,\, \log\ell_U(P) + \sum_{i=1}^k \lambda_i g_i(P),$$
where $\mathcal V_r$ is defined by the vanishing of the polynomials $g_1,\dots,g_k$.
If $P$ is a critical point of rank $r$, let $\lambda\in\C^k$ be the unique vector
such that $\nabla L(P,\lambda) = 0$.  Then, $P$ is a local maximizer if the
matrix $N^T\cdot HL(P,\lambda)\cdot N$ is negative semidefinite
where $HL(P,\lambda)$ is the Hessian of $L$ and
the columns of $N$ form a basis for the tangent space of $\mathcal V_r \times \C^k$ at
$(P,\lambda)$.

In the remainder of this section we present three concrete numerical examples.

\begin{example}
\label{ex:3x3numeric} \rm
We consider the symmetric matrix model (\ref{eq:Pmatrixsym}) for $n=3$
with the data
$$
u_{11} = 10, \,
u_{12} = 9, \,
u_{13}  = 1, \,
u_{22}  = 21,\,
u_{23} = 3,\,
u_{33}  = 7. $$
All six critical points of the likelihood function
(\ref{eq:logsymmlike}) are real and positive. They are
$$
\begin{matrix}
p_{11} & p_{12} & p_{13} & p_{22} & p_{23} & p_{33} && \,\,\,{\rm log} \,\ell_U(p) \\
0.1037 &  0.3623 & 0.0186 & 0.3179 & 0.0607 & 0.1368 &&    -82.18102 \\
0.1084 &  0.2092 & 0.1623 & 0.3997 & 0.0503 & 0.0702 &&   -84.94446 \\
0.0945 &  0.2554 & 0.1438 & 0.3781 & 0.4712 & 0.0810  &&   -84.99184 \\
0.1794 &  0.2152 & 0.0142 &  0.3052 & 0.2333 &  0.0528 && -85.14678 \\
0.1565 &  0.2627 & 0.0125 & 0.2887 & 0.2186 & 0.0609 &&    -85.19415 \\
0.1636 &  0.1517 & 0.1093 & 0.3629 & 0.1811 & 0.0312 &&   -87.95759 \\
\end{matrix}
$$
The first three points are local maxima in $\Delta_5$ and the last
three points are local minima. These six points define an extension
of degree $6$ over $\mathbb{Q}$. For instance, via {\tt Macaulay\,2} \cite{M2},
the minimal polynomial
for the last coordinate is
$9528773052286944p_{33}^{6} - 4125267629399052p_{33}^{5} + 713452955656677p_{33}^{4} - 63349419858182p_{33}^{3} + 3049564842009p_{33}^{2} - 75369770028p_{33} + 744139872$.
As we shall see in Proposition \ref{prop:galois},
the Galois group of this irreducible
polynomial is solvable, so we can
express each of the coordinates in radicals.
The last coordinate, via~{\tt RadiRoot}~\cite{Radiroot},~is

 \smallskip
 \noindent
 $ p_{33}  =
\frac{16427}{227664} + \frac{1}{12}\!\left(\zeta - \zeta^{2}\right)\omega_2 - \frac{66004846384302}{19221271018849}\omega_2^2 + \left(
  \frac{14779904193}{211433981207339} \zeta^{2}
- \frac{14779904193}{211433981207339}\zeta   \right)\omega_1\omega_2^2 + \frac{1}{2}\omega_3
$,

\smallskip
\noindent where $\zeta$ is a primitive third root of unity,
$\,\omega_1^2  =  94834811/3$, and
$$
\begin{matrix} \medskip
\omega_2^3 &  =&  \left(\frac{5992589425361}{150972770845322208}\zeta - \frac{5992589425361}{150972770845322208}\zeta^{2}\right) + \frac{97163}{40083040181952}\omega_1,
\qquad \\ \smallskip
\omega_3^2 & = &  \frac{5006721709}{1248260766912} + \left(\frac{212309132509}{4242035935404}\zeta - \frac{212309132509}{4242035935404}\zeta^{2}\right)\omega_2 - \frac{2409}{20272573168}\omega_1\omega_2
\\ & &  - \frac{158808750548335}{76885084075396}\omega_2^2
+ \left(  \frac{17063004159}{422867962414678}\zeta^{2} - \frac{17063004159}{422867962414678}\zeta \right)\omega_1\omega_2^2 .
\end{matrix}
$$
We finally note that the six critical points can be matched into three pairs
  so that (\ref{eq:strongduality}) holds: the Hadamard product of
     points  1 and 6 agree with that of points 2 and 5, and that of points 3 and 4.
Thus this example illustrates the symmetric matrix version of
Theorem~\ref{conj:Constant}. \qed
\end{example}

\begin{example}
\label{ex:4x5numeric} \rm
Let $m = 4, n = 5$ and consider the data matrix
$$
U \quad = \quad
\begin{pmatrix}
2084    &       1  &         1   &        1  &         4 \\
         4  &     23587    &       5  &         3   &        1 \\
         6   &        3  &     41224    &       3 &          2 \\
         4   &        6 &          2  &      8734    &       4
\end{pmatrix}.
$$
For $r = 2$ and $r= 3$, this instance has the expected number
$843$ of distinct complex critical points.
In both cases,  $555$ critical points are real, and $25$ of these are positive.
Consider  the $25$ critical points in $\Delta_{19}$. For $r = 2$ precisely seven  are local maxima,
and  for $r = 3$ precisely six are local maxima.
We shall list them explicitly in Examples \ref{ex:em45} and \ref{ex:em45b} respectively.
\qed
\end{example}

\begin{example}
\label{ex:5x5numeric} \rm
Let $m = n = 5$, with the non-symmetric model, and consider the data
$$
U \quad = \quad
\begin{pmatrix}
2864 & 6 & 6 & 3 & 3 \\
2 & 7577 & 2 & 2 & 5 \\
4 & 1 & 7543 & 2 & 4 \\
5 & 1 & 2 & 3809 & 4 \\
6 & 2 & 6 & 3 & 5685
\end{pmatrix}.
$$
For $r = 2$ and $r= 4$, this instance has the expected number of
$6776$ distinct complex critical points.
In both cases,  $1774$ of these are real and $90$ of these are
real and positive.
This illustrates the last statement in Theorem \ref{conj:Constant}.
The number of local maxima for $r = 2$ equals $15$,
and the number of local maxima for $r = 4$ equals $6$.
For $r = 3$, we have $61326$ critical points, of which
$15450$ are real. Of these, $362$ are positive and
$25$ are local maxima.
\qed
\end{example}

\section{Further results and computations} \label{sec:Conjectures}

The numerical algebraic geometry techniques described in
Section~\ref{sec:NAG} have the advantage that they permit
fast experimentation with non-trivial instances.
This led us to a variety of conjectures, including those
concerning ML duality.
Before we come to our discussion of duality, we briefly state
 a conjecture regarding the ML degree of $3 \times n$-matrices of rank $2$.

\begin{conjecture}\label{conj:3ncase}
For $m = 3$ and $n \geq 3$,
the ML degree of the variety $\,\mathcal{V}_2\,$
equals $\,2^{n+1} - 6$.
\end{conjecture}

The first three values already appeared in Theorem \ref{thm:MLvalues}.
We tested this formula by solving the equations of
the local kernel formulation (\ref{eq:formulation5}).
This was done independently in {\tt Macaulay\,2} and {\tt Bertini}.
With these computations,
we verified Conjecture \ref{conj:3ncase} up to $n = 10$.
This conjecture, if correct, would furnish a simple and natural
sequence of models, namely $3 \times n$-matrices of rank $2$,
whose ML degree grows exponentially in the number of states.

\smallskip

We next formulate a refined version of
the duality statement in  Theorem~\ref{conj:duality}.
Given a data matrix $U$ of format $m \times n$,  we write $\Omega_U$ for the $m \times n$-matrix
whose $(i,j)$ entry equals
$$ \frac{u_{ij} u_{i+} u_{+j}}{(u_{++})^3}. $$
The following statement also appeared as a conjecture in the first
version of our paper, and it was proved
by Draisma and Rodriguez in their article \cite{DR}
on maximum likelihood duality.

\begin{theorem}[\cite{DR}] \label{conj:Constant}
Fix $m \leq n$ and $U$ an $m \times n$-matrix with strictly positive integer entries.
There exists a bijection
 between the complex critical points $P_1,P_2,\ldots,P_s$
 of the likelihood function $\,\ell_U$ on $\mathcal V_r$ and
 the complex critical points $Q_1,Q_2,\ldots, Q_s$ of $\,\ell_U$ on $\mathcal V_{m-r+1}$
 such that
 \begin{equation}
\label{eq:strongduality}
 P_1 \star Q_1 \, = \,
 P_2 \star Q_2  \, = \,
\, \cdots \,\,=\,
 P_s \star Q_s \,\, = \,\,  \Omega_U.
 \end{equation}
 In particular, this bijection preserves reality, positivity, and rationality of the critical points.
 \end{theorem}

From the perspective of statistics, this result implies
the following striking statement: maximum likelihood estimation for matrices of rank $r$
is exactly the same problem as minimum likelihood estimation for matrices of corank $r-1$,
and vice versa.
This refined formulation of the duality statement allows us to improve
the speed of MLE by passing to the complementary
problem, where it may be easier to solve the likelihood equations.
We saw a first instance of this in Section \ref{sec:NAG}
when we discussed the
last two columns in Table~\ref{tab:PreprocessTime}:
the two methods give the same set of $191$ solutions but
the running times are complementary.
 
\begin{remark} \rm
Equation (\ref{eq:strongduality})
is trivially satisfied for $r = 1$, where
the ML degree is  $s = 1$. Here, $P_1$ is the rank one matrix in
(\ref{eq:rankone}),
and $Q _1  = \frac{1}{u_{++}} U$. Clearly, we have $P_1 \star Q_1 = \Omega_U$. \qed
\end{remark}

We illustrate Theorem~\ref{conj:Constant} for a specific case
that has already appeared in the literature
\cite{FHRZ,ASCB,ZJG}. The first assertion in the next theorem
resolves \cite[Conjecture~11]{ZJG} affirmatively.
In their conjecture, Zhu {\em et al.}~\cite{ZJG}
had identified the matrix $P(a,b)$ below,
and they had asserted that it is the global maximum of
the likehood function for the data matrix $U(a,b)$.
Note that, for $a = 4$ and $b= 2$, this is
the matrix for DiaNA's data in \cite[Example 1.16]{ASCB}.

\begin{theorem} \label{thm:DiaNA}
Let $m=n=4$, $a>b>0$, and consider the following matrices:
$$U(a,b)\, =\, \left[\begin{array}{cccc}
a & b & b & b\\
b & a & b & b\\
b & b & a & b\\
b & b & b & a
\end{array}\right] \,\,\, \hbox{~~and~~} \quad
P(a,b) \,=\, \frac{1}{8(a+3b)}\left[\begin{array}{cccc}
a+b & a+b & 2b & 2b\\
a+b & a+b & 2b & 2b\\
2b & 2b & a+b & a+b\\
2b & 2b & a+b & a+b
\end{array}\right].$$
The distribution $P(a,b)$ maximizes the likelihood function for
the data matrix $U(a,b)$ on
$\mathcal{M}_2$.
\end{theorem}

\begin{proof}
This statement is invariant under scaling the vector $(a,b)$. 
We
normalize by taking $4a + 12b = 16$.
Then  $b = (4-a)/3$ and $a$ ranges in the open interval
defined by $1 < a < 4$. For each such $a$, the likelihood function
$\ell_{U(a,b)}$ has exactly $25$ positive critical points in the rank $2$ model
$\mathcal M_2$, with the maximum value occurring at~$P(a,b)$.
This statement was shown using the following method
and its illustration in Figure \ref{fig:Plot}.

\begin{figure}[ht!]
  \centering
  \includegraphics[scale=0.4]{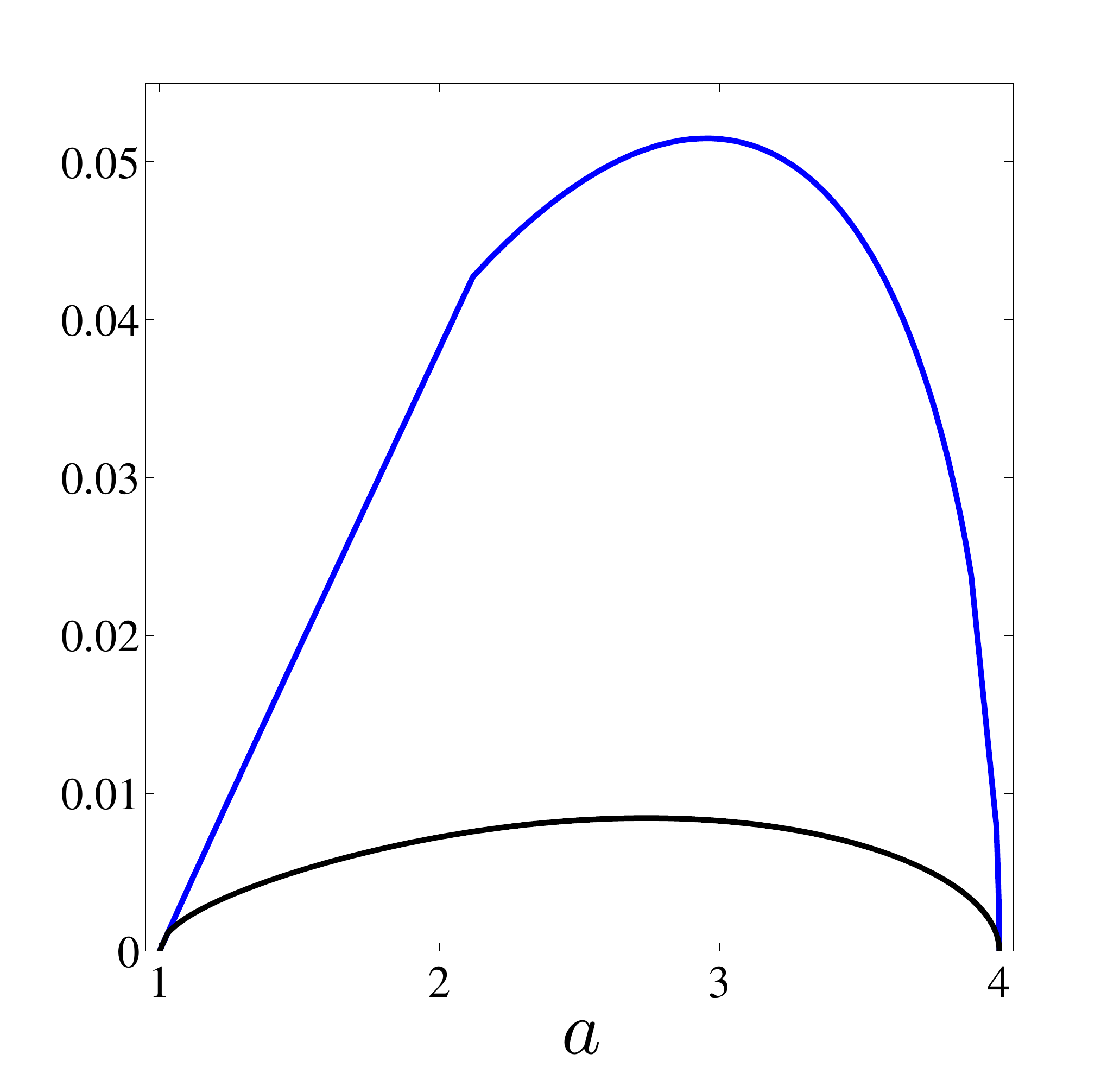}\\
  \caption{Minimum pairwise distance and lower bound \eqref{eq:lowerbound} as a function of $a$}\label{fig:Plot}
\end{figure}

First, we selected $a = 2$ and computed the $191$ critical points
using {\tt Bertini}.  From these, {\tt alphaCertified}
proved that exactly $25$ are real and, using the computed
error bounds, it verified that all lie in $\Delta_{15}$.
We then expressed these real solutions as rational functions in $a$ and $b$
to show that all $25$ real solutions remain positive for all $a > b > 0$.
The critical points fall into four symmetry classes of size $6$, $12$, $4$, and $3$.
Representatives of these classes are
$$
\begin{small}
\begin{array}{ll}
X_1\, =\, \dfrac{1}{16} \left[\begin{array}{cccc} 1 & 1 & 1 & 1 \\ 1 & 1 & 1 & 1 \\ 1 & 1 & \frac{2a}{a+b} & \frac{2b}{a+b} \\ 1 & 1 & \frac{2b}{a+b} & \frac{2a}{a+b} \end{array}\right], \qquad
X_2 \,=\, \dfrac{1}{32(a+2b)}
\left[\begin{array}{cccc} 2a+4b & 2a+4b & 2a+4b & 2a+4b
 \\ 2a+4b & 6a& 6b & 6b \\ 2a+4b & 6b& 3a+3b & 3a+3b\\
 2a+4b & 6b & 3a+3b & 3a+3b\end{array}\right], \\ \\
X_3 \,=\, \dfrac{1}{12(a+3b)} \left[\begin{array}{cccc} 3a & 3b & 3b & 3b \\ 3b & a+2b & a+2b & a+2b \\ 3b & a+2b & a+2b & a+2b \\ 3b & a+2b & a+2b & a+2b \end{array}\right],
\qquad \hbox{and} \quad
X_4 \,=\, P(a,b). \end{array}
\end{small}
$$
Using calculus, one can prove that
$\log\ell_U(X_i) < \log\ell_U(X_{i+1})$ for $i = 1,2,3$.

All that remains is to show that the $191$ solutions remain distinct
on $1 < a < 4$ (with some coalesce at the boundary).
The function mapping $a$ to the minimum
of the pairwise distances between the critical points is a
piecewise smooth function.  It is depicted in Figure~\ref{fig:Plot}.
By tracking the homotopy paths as $a$ changes from $2$ to $1$ and from $2$ to $4$,
we are able to determine that this function is nowhere zero on the open interval $(1,4)$.
Additionally, by analyzing the solutions using \cite{BHMPS}, a lower bound on this
minimum pairwise distance function~is
\begin{equation}\label{eq:lowerbound}
\min\left\{\begin{array}{ll} \frac{(a-1)\sqrt{a^2+17}}{12(a+8)}, & \frac{a+2-\sqrt{(a-1)(a-4)}}{48} - \frac{3(a^2-12a-16)+\sqrt{6(a-1)(a-4)(a^2-16a+96)}}{16(a+8)(a-10)} \end{array}\right\}
\end{equation}
which is also depicted in Figure~\ref{fig:Plot}.
The first term of this minimum arises from $X_2$ and a member of the $X_3$ family
which is equal to the minimum pairwise distances for values of $a$ near $1$.
The second term arises from comparing the $(1,1)$ entries of critical points.
In short, all of the solutions remain distinct on $1 < a < 4$
and this establishes \cite[Conjecture~11]{ZJG}.
\end{proof}

We checked the duality statement in
Theorem~\ref{conj:Constant} by performing
the same computation for $m=n=4$ and $r = 3$.
We followed the $191$ paths in the deformation from a
general $U_0$ to a general $U(a,b)$.  Using {\tt Bertini}, we found that $12$
endpoints had rank $2$ while the other $179$ had the expected rank of $3$.
Moving the other $179$ solutions to $a = 2$ produced $179$ distinct
complex solutions that remain distinct and retain rank $3$ on $(1,4)$.
Using the same certification process as above,
precisely $25$ are positive.
These critical points of $\mathcal M_3$
form four symmetry classes having the same sizes
$6,12,4$, and $3$ as above, with representatives:
$$
\begin{small}
\begin{array}{ll}
Y_1 = \frac{1}{8(a+3b)} \left[\begin{array}{cccc} 2a & 2b & 2b & 2b \\ 2b & 2a & 2b & 2b \\ 2b & 2b & a+b & a+b \\ 2b & 2b & a+b & a+b \end{array}\right],  \qquad
Y_2 = \frac{1}{12(a+3b)} \left[\begin{array}{cccc} 3a & 3b & 3b & 3b \\ 3b & a+2b & a+2b & a+2b \\ 3b & a+2b & \frac{2a(a+2b)}{a+b} & \frac{2b(a+2b)}{a+b} \\ 3b & a+2b & \frac{2b(a+2b)}{a+b} & \frac{2a(a+2b)}{a+b} \end{array}\right], \\ \\
Y_3 = \frac{1}{16(a+2b)} \left[\begin{array}{cccc} a+2b & a+2b &\! a+2b & \! a+2b \\ a+2b & 3a & 3b & 3b \\ a+2b & 3b & 3a & 3b \\ a+2b & 3b & 3b & 3a \end{array}\right] \! , \,\,\,
Y_4 = \frac{1}{16(a+b)} \left[\begin{array}{cccc} 2a & 2b & a+b & a+b \\ 2b & 2a & a+b & a+b \\ a+b & a+b & 2a & 2b \\ a+b & a+b & 2b & 2a\end{array}\right]. \\
\end{array}
\end{small}
$$
The matrices are now sorted by decreasing value of $\ell_{U(a,b)}$, so the
first matrix $Y_1$ is the MLE.
Our real positive critical points satisfy the desired duality relation. Namely, we have
$$  X_1 \star Y_1 \,= \,
 X_2 \star Y_2 \,= \,     X_3 \star Y_3 \,= \,     X_4 \star Y_4 \,\,= \,  \,
 \frac{1}{64(a{+}3b)} U(a,b) \,=: \, \Omega_U . $$
 We verified the same for the complex solutions.

When Theorem \ref{conj:Constant} was still a conjecture, we
verified it for randomly selected data matrices
with i.i.d.~entries sampled from the uniform distribution on $[0,1]$.
After generating a random matrix, we verified
equation (\ref{eq:strongduality}) using
the critical points computed by homotopy continuation.
For $m = n = 3$ and $r = 2$,
we verified (\ref{eq:strongduality}) for 50000 instances.
Additionally, for $m = n = 4$ and $r = 2$,
we verified (\ref{eq:strongduality}) for 10000~instances.
We also did this for a handful of $4 \times 5$ instances (such
as Example~\ref{ex:4x5numeric}) and $5 \times 5$ instances (such as
 Example~\ref{ex:5x5numeric}).
The user can find  {\tt Macaulay\,2} code, which uses the emerging
 {\tt Bertini.m2} package, to perform more numerical experiments at
{\tt math.berkeley.edu/\text{\textasciitilde}jrodrig/code/rankConstraints}.

Theorem \ref{conj:Constant} and its analogue for symmetric matrices
is particularly interesting in the special case when $m=n = 2r-1$.
Here we have
 an involution on the set of critical points
of $\ell_U$ on $\mathcal{V}_r$ which has the following property.
If $P_1, P_2, \ldots, P_s$ are the positive critical points in the model
$\mathcal{M}_r$, ordered by increasing value of the log-likelihood function, then
$$ \ell_U(P_1) + \ell_U(P_s) \, = \,
 \ell_U(P_2) + \ell_U(P_{s-1}) \, = \,\,\cdots \,\, = \,
\ell_U(P_{\lceil{s/2} \rceil})  +
\ell_U(P_{\lfloor{s/2} \rfloor}).
 $$
The identity (\ref{eq:strongduality}) implies
that Galois group which permutes the
set of critical points is considerably smaller than the full symmetric group
on these points.
We shall demonstrate this for $n=3$.
 What follows will explain the solutions in radicals
seen in Example~\ref{ex:3x3numeric}.

Let  $\mathbb{Q}(U)$ denote the field of rational functions
in entries of an indeterminate data matrix $U$, and let
$K$ denote the algebraic extension of $\mathbb{Q}(U)$
that is defined by adjoining all solutions of the likelihood equations.
Thus the degree of the extension $K/\mathbb{Q}(U)$ is the ML degree.
We are interested in the Galois group $G = {\rm Gal}(K,\mathbb{Q}(U))$
of this algebraic extension. This Galois group is a subgroup of
the full symmetric group $S_M$ where $M$ is the ML degree.

The following result was found by explicit computations using
{\tt maple} and {\tt Sage} \cite{sage}.

\begin{proposition}
\label{prop:galois}
The Galois group for MLE on $3 \times 3$-matrices (\ref{eq:Pmatrix}) of rank $2$
is a subgroup of order $1920$ in $S_{10}$.
As an abstract group, it is the semidirect product of $S_5$ and $(\mathbb{Z}_2)^4$.
The Galois group for MLE on symmetric $3 \times 3$-matrices
(\ref{eq:Pmatrixsym}) of rank $2$ is a subgroup of order $24$ in $S_5$.
As an abstract group, it is the symmetric group $S_4$. So, in the latter case, the six
critical points of the likelihood function can be written in radicals
in $\,u_{11}, u_{12}, u_{13}, u_{22}, u_{23}, u_{33}$.
\end{proposition}

We close this section with an important observation
 that is implied by the various polynomial formulations of our problem,
but which had not been explicitly stated
in Section~\ref{sec:Equations}.

\begin{remark}
Every complex critical point $ P$ of the likelihood function $\ell_U$ on $\mathcal{V}_r$ satisfies
$$ p_{i+} = \frac{u_{i+}}{u_{++}}  \,\,\,\,\hbox{for}\,\,\, i = 1,\ldots,m
\qquad \hbox{and} \qquad
p_{+j} = \frac{u_{+j}}{u_{++}} \, \,\,\,\hbox{for}\,\,\, j = 1,\ldots,n . $$
\end{remark}

The analogous identities hold for any statistical model
that is {\em toric} in the sense of \cite{ASCB}.
Namely, the critical points of the likelihood function on any
secant variety of a toric variety have the sufficient statistics
of the given data in the toric model. This fact seems relevant for the topological
underpinnings of ML duality. One is tempted to speculate that
some version of Theorems \ref{conj:duality}, \ref{conj:dualitysym}, and
\ref{conj:Constant}  might be true for other classes of toric models.

\section{Rank versus non-negative rank}\label{sec:Stats}

In the previous sections, we developed accurate methods for finding the
global maximum of a likelihood function $\ell_U$
over non-negative matrices $P$ of rank $r$ whose entries sum to~$1$.
Unfortunately, this is not quite the problem most practitioners
and users of statistics would actually be interested in.
Rather than restricting the rank  of a probability table (\ref{eq:Pmatrix}),
it is the {\em non-negative rank} that is more relevant for applications.
In this section we discuss this.

Let ${\rm Mix}_r$ denote the subset of $\Delta_{mn-1}$
that comprises all the mixtures of $r$ independent distributions.
In statistics, this is the archetype of a latent variable model,
or hidden variable model.
Mathematically, we can define the {\em mixture model}
$\,{\rm Mix}_r$ as the set of all matrices
\begin{equation}
\label{eq:mixequation}
 P \,\,=\,\, A \cdot \Lambda \cdot B,
 \end{equation}
where $A$ is a non-negative $m \times r$-matrix whose
columns sum to $1$,
$\Lambda$ is an $r \times r$ diagonal matrix
whose diagonal entries are non-negative and sum to $1$,
and $B$ is a non-negative $r \times n$-matrix whose
rows sum to $1$. The {\em rank-constrained model}
$\mathcal{M}_r = \mathcal{V}_r \cap \Delta_{mn-1}$ we discussed above is an algebraic relaxation
of the mixture model ${\rm Mix}_r$.
This can be made precise~as~follows:

\begin{proposition}
\label{prop:mixture}
The rank-constrained model $\mathcal{M}_r$ is the Zariski closure
of the mixture model ${\rm Mix}_r$ inside the simplex $\Delta_{mn-1}$.
If $\,r \leq 2\,$ then $\,{\rm Mix}_r = \mathcal{M}_r$.
If $\,r \geq 3\,$ then $\,{\rm Mix}_r \subsetneq \mathcal{M}_r$.
\end{proposition}

\begin{proof}
See Example 4.1.2, Example 4.1.4 and Proposition 4.1.6 in \cite{LiAS}.
That book refers to secant varieties of Segre varieties, tensors of any format,
 and joint distributions of any number of random variables.
 Here we only need the case of matrices and two random variables.
\end{proof}

Our model  $\mathcal{M}_r$ is the set of all distributions $P$
of rank at most $r$, while
${\rm Mix}_r$ is the set of all distributions $P$ of non-negative rank at most $r$.
Having non-negative rank $\leq r$ means that
$P = A' \cdot B'$ for some non-negative matrices
where $A'$ has $r$ columns and $B'$ has $r$ rows.
Any such factorization can be transformed into the particular
form (\ref{eq:mixequation}) which identifies the statistical parameters.
For further information on these two models see \cite{FHRZ, MSS, ASCB}.

\smallskip

Understanding the inclusion of ${\rm Mix}_r$ inside $\mathcal{M}_r$
becomes crucial when comparing different methodologies for
maximum likelihood estimation. We used {\tt Bertini} to compute
all critical points of the likelihood function $\ell_U$
on $\mathcal{M}_r$, with the aim of identifying the
global maximum $\widehat P$ of $\ell_U$ over $\mathcal{M}_r$.
This assumes that $\widehat P$ is strictly positive. This is
usually the case when $U$ is strictly positive.
The standard method used by statisticians is to run the
{\em EM algorithm} in the space of model parameters $(A,\Lambda, B)$.
This results in a local maximum $(\widehat A, \widehat \Lambda, \widehat B)$
of the likelihood function expressed in terms of the parameters.
The fact that $\mathcal{M}_r$ is the Zariski closure
of the mixture model ${\rm Mix}_r$ in the simplex $\Delta_{mn-1}$ has the following~consequence:

\begin{corollary}
\label{cor:nonneg}
Let $\widehat P_1, \ldots, \widehat P_s$ be
the local maxima in $\mathcal{M}_r$ of the likelihood function $\ell_U$.
If a matrix $\widehat P_i$ has non-negative rank at most $r$
then $\widehat P_i$ lies in ${\rm Mix}_r$ and
matching parameters
$(\widehat A_i, \widehat \Lambda_i, \widehat B_i)$ can found by solving
(\ref{eq:mixequation}).
If all matrices $\widehat P_i$ have non-negative rank strictly larger than $r$
then $\ell_U$ attains its maximum over ${\rm Mix}_r$
on the topological boundary $\partial {\rm Mix}_r$.
\end{corollary}

\begin{proof}
The second sentence holds because
 every matrix $P \in \Delta_{mn-1}$ of non-negative rank $\leq r$ admits
a factorization of the special form (\ref{eq:mixequation}).
Indeed, if $P = A' \cdot B'$ is any non-negative factorization then
we first scale the rows of $A'$ to get a matrix $A$ with row sums equal to $1$,
and we adjust the second matrix so that $P = A \cdot B''$.
Now let $\Lambda$ be the diagonal matrix whose entries
are the column sums of $B''$ and set
$B =  \Lambda^{-1} B''$. Then $P = A \cdot \Lambda \cdot B$.

For the third sentence, suppose
$\ell_U$ has its maximum over ${\rm Mix}_r$
at a point $\widehat P$ in ${\rm Mix}_r \backslash \partial {\rm Mix}_r$.
Then $\widehat P$ is also a local maximum of $\ell_U$ on
$\mathcal{M}_r$. Thus $\widehat P$ will  be found by solving the critical equations for $\ell_U$ on
$\mathcal{V}_r$. The matrix $\widehat P$ is an element of
$\{\widehat P_1, \ldots , \widehat P_s \}$. Hence, this set
contains a matrix of non-negative rank $\leq r$.
This proves the contrapositive of the assertion.
\end{proof}

We shall now discuss
the exact solution of the MLE problem for
the mixture model ${\rm Mix}_r$.
Let us start with the low rank cases.
The given input is a  data matrix $U$ as in (\ref{eq:Umatrix}).

If $r=1$ then the likelihood function $\ell_U$ has a unique critical point.
Let $\,u_{*+}\,$ be the column vector of row sums of $U$,
and let $\,u_{+*}\,$ be the row vector of column sums of $U$.
Then
\begin{equation}
\label{eq:rankone}
\widehat P \quad  = \quad \frac{1}{(u_{++})^2} \cdot u_{*+} \cdot u_{+*}.
\end{equation}
If $r \geq 2$ then we compute the set
$\{\widehat P_1,\ldots,\widehat P_s \}$ of
all local maxima of the likelihood function
$\ell_U$ on the model $\mathcal{M}_r$. This is done
using the numerical algebraic geometry methods described in Section~\ref{sec:NAG}, by
solving the likelihood equations (\ref{eq:formulation5}) for the determinantal variety $\mathcal{V}_r$.

If $r = 2$ then every matrix $\widehat P_i$ has non-negative rank $\leq 2$.
We therefore select the matrix whose likelihood value
$\ell_U(\widehat P_i)$ is maximal. Then $\widehat P_i$ solves
the MLE problem for ${\rm Mix}_2 = \mathcal{M}_2$.

\begin{example}
\label{ex:em45} \rm
We experimented with the EM Algorithm for $r=2$, as  in \cite[\S 1.3]{ASCB},
on the $4 \times 5$ data matrix $U$  discussed in
Example \ref{ex:4x5numeric}.
We ran $10000$ iterations with
starting points $(A,\Lambda,B)$ sampled from the uniform distribution
on the $15$-dimensional parameter polytope
$$ (\Delta_3 \times \Delta_3) \, \times \,\Delta_1 \,
\times \,(\Delta_4 \times \Delta_4)  . $$
From these $10000$ runs of the EM algorithm we obtained the following
seven local maxima:
$$
\begin{matrix}
\smallskip
 \,\, 2643\,\, \hbox{occurrences:}  & \begin{tiny}
\begin{bmatrix}
       0.001678 & 0.01892 & 0.00001325 & 0.007008 & 0.00000722  \\
         0.01894 & 0.2136 & 0.00006605 & 0.07912 & 0.00008149\\
      0.00007930 & 0.00003964 & 0.5447 & 0.00003964 & 0.00002643\\
        0.007023 & 0.07921 & 0.00002643 & 0.02933 & 0.00003021
\end{bmatrix}
\end{tiny} & {\rm log}(\ell_U) \, = \,  -105973.49 \\
\smallskip
2044 \,\, \hbox{occurrences:}  & \begin{tiny}
\begin{bmatrix}
     0.001332 & 0.00001777 & 0.02627 & 0.00000792   & 0.00000382 \\
       0.00007696 & 0.2274 & 0.00006503 & 0.08423 & 0.00004823\\
       0.02628 & 0.00003913 & 0.5185 & 0.00004103 & 0.00007542\\
       0.00002871 & 0.08432 & 0.00002762 & 0.03123 & 0.00001788
\end{bmatrix}
\end{tiny} & {\rm log}(\ell_U) \, = \,     -106487.35 \\
\smallskip
1897 \,\, \hbox{occurrences:}  & \begin{tiny}
\begin{bmatrix}
     0.002245 & 0.02536 & 0.00001725 & 0.000006332   & 0.000005379 \\
       0.02535 & 0.2863 & 0.00006471 & 0.00004393 & 0.00006072\\
       0.00009818 & 0.00003897 & 0.4495 & 0.09525 & 0.00006537\\
       0.00002773 & 0.00008630 & 0.09530 & 0.02020 & 0.00001388
\end{bmatrix}
\end{tiny} & {\rm  log}(\ell_U) \, = \,    -109697.04  \\
\smallskip
1688 \,\,  \hbox{occurrences:}  & \begin{tiny}
\begin{bmatrix}
       0.001111 & 0.00001327 & 0.02187 & 0.004634 & 0.000005304 \\
      0.00005289 & 0.3117 & 0.00006605 & 0.00003968 & 0.00001322\\
         0.02191 & 0.00003963 & 0.4314 & 0.09144 & 0.0001046\\
        0.004647 & 0.00007931 & 0.09148 & 0.01939 & 0.00002219
\end{bmatrix}
\end{tiny} & {\rm log}(\ell_U) \, = \,    -111172.67   \\
\smallskip
1106 \,\, \hbox{occurrences:}  & \begin{tiny}
\begin{bmatrix}
      0.005321 & 0.00002006 & 0.00001106 & 0.02226 & 0.00002038\\
        0.00005070 & 0.1135 & 0.1983 & 0.00004009 & 0.00001444\\
        0.00008126 & 0.1983 & 0.3465 & 0.00003939 & 0.00002520\\
       0.02227 & 0.00007333 & 0.00002771 & 0.09316 & 0.00008532
\end{bmatrix}
\end{tiny} & {\rm log}(\ell_U) \, = \,    -127069.50  \\
\smallskip
529 \,\, \hbox{occurrences:}  & \begin{tiny}
\begin{bmatrix}
      0.0008641 & 0.009735 & 0.01701 & 0.00001350 & 0.00000289 \\
         0.009756 & 0.1099 & 0.1921 & 0.00003965 & 0.00003259\\
         0.01705 & 0.1921 & 0.3357 & 0.00003959 & 0.00005693\\
      0.00005301 & 0.00007930 & 0.00002642 & 0.1154 & 0.00005294
\end{bmatrix}
\end{tiny} & {\rm log}(\ell_U) \, = \,  -131013.73    \\
93 \,\, \hbox{occurrences:}  & \begin{tiny}
\begin{bmatrix}
     0.02754 & 0.00001320 & 0.00001319 & 0.00001334 & 0.00005311\\
         0.00005280 & 0.09999 & 0.1747 & 0.03704 & 0.00002957\\
         0.00007916 & 0.1747 & 0.3053 & 0.06472 & 0.00005164\\
        0.00005339 & 0.03706 & 0.06476 & 0.01373 & 0.00001102
\end{bmatrix}
\end{tiny} & {\rm log}( \ell_U) \, = \,   -148501.63   \\
\end{matrix}
$$
The first matrix is the global maximum, and it was the output in
$2643$ of our $10000$ runs.
Note that the ordering by objective function value
agrees with the ordering by occurrence.
We know from Example \ref{ex:4x5numeric} that
$\Delta_{19} $ contains $7$ local maxima, and hence our EM experiment found them all.
Each of the $7$ matrices above has both rank and non-negative rank $r=2$.
\qed
\end{example}

If $r \geq 3$ then the situation is more challenging.
To begin with,  we need a method for testing whether
a matrix has non-negative rank $\leq r$.
Recent work by Moitra \cite{Moi}
shows that the computational complexity of this
problem is lower than one might fear at first glance.

So, let us assume for now that this problem has been solved
and we have an algorithm to decide quickly
whether any of the matrices $\widehat P_i$
has non-negative rank $r$. If so, we pick among them
the matrix $\widehat P_i$ of largest $\ell_U$-value.
This matrix is now a candidate for the MLE on ${\rm Mix}_r$.
But it  may not actually be the MLE because the
global maximum of the likelihood function  $\ell_U$
may be attained on the boundary $\partial {\rm Mix}_r$.
Furthermore, it is quite possible that none of the critical points
in $\{\widehat P_1, \ldots , \widehat P_s \}$ lies in ${\rm Mix}_r$.
Then, according to the third sentence of Corollary
\ref{cor:nonneg}, the MLE in the mixture model ${\rm Mix}_r$
necessarily lies in the boundary  $\partial {\rm Mix}_r$.

Our discussion implies that, in order to perform exact
maximum likelihood estimation for the mixture model,
we need to have an exact algebraic description
of $\partial {\rm Mix}_r$. Specifically, we must
determine the polynomial equations that cut out the various irreducible components of
the Zariski closure of $\partial {\rm Mix}_r$ as a subvariety of $\PP^{mn-1}$.
For each of these components, and the various strata where they intersect,
we then need to compute the ML degree.
That list of further ML degrees,
 combined with the value for $\mathcal{V}_r$ in
 Theorem \ref{thm:MLvalues}, describes the true intrinsic
algebraic complexity of the MLE $\widehat P$
as a piecewise algebraic function of the data $U$.

To be even more ambitious, we could ask for
an exact semi-algebraic description of the
set ${\rm Mix}_r$. Namely, what we seek is a
Boolean combination of polynomial inequalities in
the unknowns $p_{ij}$ that characterize ${\rm Mix}_r$ as a subset of
  $\mathcal{V}_r \cap \Delta_{mn-1}$.
Finding such a description is an open problem, even in the
small cases that are covered by Theorem \ref{thm:MLvalues}.
We believe that it might be possible to resolve the problem
for these cases, where $(m,n,r)$ ranges from $(4,4,3)$ to $(5,5,4)$,
using the techniques developed by Mond, Smith, and van Straten in \cite{MSS}.

We  illustrate the proposed approach for the first interesting case $(m,n,r) = (4,4,3)$.
Components of $\partial {\rm Mix}_3$ correspond to different labelings of the
configurations in \cite[Figure 9]{MSS}. Using the translations (seen
in \cite[\S 2]{MSS})  between non-negative  factorizations (\ref{eq:mixequation}) and nested polygons,
one of the  labelings of \cite[Figure 9 (a)]{MSS} corresponds to the factorization
\begin{equation}
\label{eq:P=AB}
\begin{pmatrix}
p_{11} & p_{12} & p_{13}  & p_{14} \\
p_{21} & p_{22} & p_{23}  & p_{24} \\
p_{31} & p_{32} & p_{33}  & p_{34} \\
p_{41} & p_{42} & p_{43}  & p_{44}
\end{pmatrix}
 \quad = \quad
\begin{pmatrix}
0  & a_{12}  & a_{13} \\
0 & a_{22} & a_{23} \\
a_{31} & 0 & a_{33} \\
a_{41} & a_{42}  & 0
\end{pmatrix} \cdot
\begin{pmatrix}
0 & b_{12} & b_{13} & b_{14} \\
b_{21} & 0 & b_{23}  & b_{24} \\
b_{31} & b_{32} & b_{33} & 0
\end{pmatrix}.
\end{equation}
This equation parametrizes an irreducible
divisor in the $14$-dimensional variety
$\mathcal{V}_3 \subset \PP^{15}$.
That divisor is one of the irreducible components
of the algebraic boundary of $\mathcal{M}_3$.
The corresponding prime ideal of height $2$
in $\mathbb{Q}[p_{11}, \ldots,p_{44}]$ is obtained by
eliminating the $17$ unknowns $a_{ij}$ and $b_{ij}$
from the $16$ scalar equations in (\ref{eq:P=AB}).
We find that this ideal is generated by the $4 \times 4$-determinant
that defines $\mathcal{V}_3$
together with four sextics such as
$$
\begin{matrix}
 p_{11} p_{21} p_{22} p_{32} p_{33} p_{43}
-p_{11} p_{21} p_{22} p_{33}^2 p_{42}
-p_{11} p_{21} p_{23} p_{32}^2 p_{43}
+p_{11} p_{21} p_{23} p_{32} p_{33} p_{42}
-p_{11} p_{22}^2 p_{31} p_{33} p_{43} \\
+p_{11} p_{22} p_{23} p_{31} p_{32} p_{43}
+p_{11} p_{22} p_{23} p_{31} p_{33} p_{42}
-p_{11} p_{23}^2 p_{31} p_{32} p_{42}
+p_{12} p_{21} p_{22} p_{33}^2 p_{41}
{-}p_{12} p_{21} p_{23} p_{32} p_{33} p_{41} \\
-p_{12} p_{22} p_{23} p_{31} p_{33} p_{41}
+p_{12} p_{23}^2 p_{31} p_{32} p_{41}
+p_{13} p_{21}^2 p_{32}^2 p_{43}
-p_{13} p_{21}^2 p_{32} p_{33} p_{42}
-2 p_{13} p_{21} p_{22} p_{31} p_{32} p_{43}\\
+p_{13} p_{21} p_{22} p_{31} p_{33} p_{42}
+p_{13} p_{21} p_{23} p_{31} p_{32} p_{42}
+p_{13} p_{22}^2 p_{31}^2 p_{43}
-p_{13} p_{22} p_{23} p_{31}^2 p_{42}.
\end{matrix}
$$
What needs to be studied now is the ML degree
of this codimension $2$ subvariety of $\PP^{15}$, and
the approach of \cite{Huh} would lead us to look at the topology
of the associated very affine~variety.

\smallskip

Described above is the geometry of the MLE problem
for the mixture model  ${\rm Mix}_r $ regarded as a subset
 of the ambient simplex $\Delta_{mn-1}$.
Statisticians, on the other hand, are more accustomed to working in the
space of model parameters, which is the product of simplices
\begin{equation}
\label{eq:paraspace}
(\Delta_{m-1})^r \times \Delta_{r-1} \times (\Delta_{n-1})^r.
\end{equation}
Here our parameters are $(A, \Lambda, B)$.
The model ${\rm Mix}_r$ is the image of this parameter space in $\Delta_{mn-1}$
under the map (\ref{eq:mixequation}).
That parametrization is very far from
 identifiable. The reason is that the fibers of
$ (A,\Lambda,B) \mapsto P$ are  semi-algebraic sets
of possibly large dimension. In fact, the whole point of the paper \cite{MSS}
is to study the topology of these fibers as $P$ varies.

The expectation-maximization (EM) algorithm is the local method of choice
for finding the MLE on the mixture model ${\rm Mix}_r$. Our readers might enjoy
the exposition given in \cite[\S 1.3]{ASCB}. We emphasize
that the EM algorithm operates entirely in the parameter space
(\ref{eq:paraspace}). The likelihood function
$\ell_U$ pulls back to a function on the interior of (\ref{eq:paraspace}).
The EM algorithm is an iterative method that converges to a critical
point of that function, and, under some mild regularity hypotheses, that
critical point $(\widehat A, \widehat \Lambda, \widehat B)$ is then a local maximum.
The image $\widehat P$ of the point in ${\rm Mix}_r$ is then
a candidate for the global maximum of $\ell_U$ on ${\rm Mix}_r$.

\begin{example}
\label{ex:em45b} \rm
We tried the EM Algorithm also for $r=3$
on the $4 \times 5$ data matrix $U$ in
Examples \ref{ex:4x5numeric} and \ref{ex:em45}.
We ran $10000$ iterations with
starting points sampled from the uniform distribution
on the $23$-dimensional parameter polytope
$\, (\Delta_3)^3 \,
 \times \,\Delta_2  \times (\Delta_4)^3$.
From these $10000$ runs of the EM algorithm,
$9997$ converged to one of
eight local~maxima. Three of the runs led to other fixed points.
The following six local maxima are precisely the solutions already found in
Example \ref{ex:4x5numeric}.
We note that, in this particular instance, it happened that
all local maxima in the rank  model $\mathcal{M}_3$ actually
lie in ${\rm Mix}_3$, i.e.~they have non-negative rank $3$:
$$
\begin{matrix}
\smallskip
 \,\, \,\, \hbox{3521 occurrences:}  & \begin{tiny}
\begin{bmatrix}
     0.005321 & 0.00001322 & 0.00001322 & 0.02226 & 0.00002039\\
     0.00005285 & 0.3117 & 0.00006607 & 0.00003964 & 0.00001321\\
     0.00007929 & 0.00003964 & 0.5447 & 0.00003964 & 0.00002643\\
      0.02227 & 0.00007927 & 0.00002642 & 0.09316 & 0.00008532
\end{bmatrix}
\end{tiny} & {\rm log}(\ell_U) \, = \,  -84649.67679 \\

\smallskip
 \,\, \,\, \hbox{2293 occurrences:}  & \begin{tiny}
\begin{bmatrix}
     0.002244 & 0.02535 & 0.00001324 & 0.00001333 & 0.0000054 \\
      0.02535 & 0.2863 & 0.00006606 & 0.00003961 & 0.00006065\\
     0.00007929 & 0.00003964 & 0.5447 & 0.00003964 & 0.00002643\\
     0.00005291 & 0.00007928 & 0.00002643 & 0.1154 & 0.00005289
\end{bmatrix}
\end{tiny} & {\rm log}(\ell_U) \, = \,   -86583.69000 \\

\smallskip
 \,\, \,\, \hbox{1678 occurrences:}  & \begin{tiny}
\begin{bmatrix}
     0.001332 & 0.00001326 & 0.02627 & 0.00001341 & 0.0000038 \\
     0.00005289 & 0.3117 & 0.00006607 & 0.00003964 & 0.00001322\\
      0.02628 & 0.00003963 & 0.5185 & 0.00003961 & 0.00007538\\
     0.00005296 & 0.00007928 & 0.00002642 & 0.1154 & 0.00005292
\end{bmatrix}
\end{tiny} & {\rm log}(\ell_U) \, = \,  -87698.20128 \\

\smallskip
 \,\, \,\, \hbox{1320 occurrences:}  & \begin{tiny}
\begin{bmatrix}
    0.02754 & 0.00001320 & 0.00001321 & 0.00001326 & 0.00005298\\
      0.00005277 & 0.2274 & 0.00006606 & 0.08423 & 0.00004806\\
     0.00007928 & 0.00003964 & 0.5447 & 0.00003964 & 0.00002643\\
      0.00005310 & 0.08430 & 0.00002643 & 0.03122 & 0.00001788\\
\end{bmatrix}
\end{tiny} & {\rm log}(\ell_U) \, = \,  -98171.25551 \\

\smallskip
 \,\, \,\, \hbox{576 occurrences:}  & \begin{tiny}
\begin{bmatrix}
    0.02754 & 0.00001321 & 0.00001320 & 0.00001330 & 0.00005305\\
     0.00005285 & 0.3117 & 0.00006605 & 0.00003968 & 0.00001322\\
      0.00007916 & 0.00003964 & 0.4495 & 0.09526 & 0.00006519\\
      0.00005324 & 0.00007932 & 0.09528 & 0.02019 & 0.00001389
\end{bmatrix}
\end{tiny} & {\rm log}(\ell_U) \, = \,  -102495.4349 \\

\smallskip
 \,\, \,\, \hbox{68 occurrences:}  & \begin{tiny}
\begin{bmatrix}
    0.02754 & 0.00001322 & 0.00001321 & 0.00001321 & 0.00005285\\
       0.00005287 & 0.1135 & 0.1983 & 0.00003968 & 0.00001444\\
       0.00007927 & 0.1983 & 0.3465 & 0.00003962 & 0.00002520\\
     0.00005285 & 0.00007930 & 0.00002642 & 0.1154 & 0.00005285
\end{bmatrix}
\end{tiny} & {\rm log}(\ell_U) \, = \,  -121802.8945 \\
\end{matrix}
$$
In addition, our runs of the EM algorithm discovered the two local maxima
$$
\begin{matrix}
\smallskip
 \,\, \,\, \hbox{488 occurrences:}  & \begin{tiny}
\begin{bmatrix}
      0.001678 & 0.01892 & 0.00001325 & 0.007008 & 0.0000072 \\
        0.01894 & 0.2136 & 0.00006605 & 0.07912 & 0.00008149\\
     0.00007930 & 0.00003964 & 0.5447 & 0.00003964 & 0.00002643\\
       0.007023 & 0.07921 & 0.00002643 & 0.02933 & 0.00003021
\end{bmatrix}
\end{tiny} & {\rm log}(\ell_U) \, = \,  -105973.4859 \\

\smallskip
 \,\, \,\, \hbox{53 occurrences:}  & \begin{tiny}
\begin{bmatrix}
      0.001111 & 0.00001341 & 0.02187 & 0.004634 & 0.0000053 \\
     0.00005299 & 0.3117 & 0.00006602 & 0.00003976 & 0.00001324\\
        0.02191 & 0.00003960 & 0.4314 & 0.09144 & 0.0001046\\
       0.004647 & 0.00007935 & 0.09148 & 0.01939 & 0.00002219
\end{bmatrix}
\end{tiny} & {\rm log}(\ell_U) \, = \,  -111172.6663 \\
\end{matrix}
$$
These do  not satisfy the likelihood equations.
They are located on the boundary of ${\rm Mix}_3$.
\qed
\end{example}

It would be very interesting to carefully analyze the (algebraic) geometry
of the EM algorithm, even in the small cases of Theorem \ref{thm:MLvalues}.
A comparison with the methods introduced in this paper will then allow us to
ascertain  the conditions under which EM finds the global maximum, as it did in
Example \ref{ex:em45b}.
A project by Elina Robeva on this topic is under way.

\bigskip \medskip

\noindent {\bf Acknowledgments.}
We thank Serkan Ho\c sten for helpful comments.
The authors were supported by the National Science Foundation
(DMS-1262428, DMS-0943745, DMS-0968882).

\bigskip
\bigskip

\begin{small}

\noindent
Authors' adresses:

\bigskip

\noindent
Jonathan Hauenstein,
Department of Mathematics, North Carolina State University, Raleigh, NC 27695, USA,
{\tt hauenstein@ncsu.edu}

\medskip

\noindent
Jose Rodriguez and Bernd Sturmfels,
Department of Mathematics, University of California, Berkeley,  CA 94720, USA,
{\tt {jo.ro@berkeley.edu}}, {\tt {bernd@math.berkeley.edu}}
\end{small}

\end{document}